\documentclass[11pt]{article}
\usepackage{amsmath}
\usepackage{graphicx,psfrag,epsf}
\usepackage{enumerate}
\usepackage{psfrag,epsf}
\usepackage{url} 
\usepackage{amsmath,amssymb,amsfonts,amsthm,mathtools,mathrsfs}
\usepackage{array}
\usepackage{makecell}

\usepackage{tikz,setspace}
\usepackage{pgfplots}
\usepgfplotslibrary{fillbetween}
\pgfplotsset{compat=1.16}

\usepackage{bm,bbm}
\usepackage{color}
\usepackage{subfigure}
\usepackage{algorithm,algpseudocode}
\usepackage[symbol]{footmisc}
\usepackage{scalefnt}
\usepackage{authblk} 
\usepackage{multirow,centernot}
\usepackage[colorlinks=true, citecolor=blue, urlcolor=blue]{hyperref}
\usepackage{natbib}
\usepackage{dsfont}
\usepackage[title]{appendix}
\usepackage{newtxtext,newtxmath} 

\input cyracc.def

\usepackage[left=1in,top=1.1in,right=0.5in,bottom=1in]{geometry}

\theoremstyle{definition}

\newtheorem{theorem}{Theorem}[section]

\newtheorem{corollary}[theorem]{Corollary}

\newtheorem{lemma}[theorem]{Lemma}
\newtheorem{proposition}[theorem]{Proposition}
\newtheorem{remark}[theorem]{Remark}

\makeatletter
\def\@seccntformat#1{\@ifundefined{#1@cntformat}%
	{\csname the#1\endcsname\quad}
	{\csname #1@cntformat\endcsname}
}
\makeatother

\newif\ifShowComments
\ShowCommentstrue
\def\strutdepth{\dp\strutbox}
\def\druk#1{\strut\vadjust{\kern-\strutdepth
        {\vtop to \strutdepth{%
                \baselineskip\strutdepth\vss
                        \llap{\hbox{#1}\quad}\null}}}}

\title{\bf
Exact finite-sample bias, variance, and MSE of normalized scale-invariant inequality indices under Gamma mixture models
}

\author{
\text{Roberto Vila}$^{1}$\thanks{Corresponding author: Roberto Vila, email: {rovig161@gmail.com}
}
\,\, and
\text{Felipe Quintino}$^{1}$
\\
{\small $^{1}$ Department of Statistics, University of Brasilia, Brasilia, Brazil}\\
}
\begin{document}
	\maketitle 	
	\begin{abstract}
We study the finite-sample properties of estimators of a broad class of normalized scale-invariant inequality indices (NSIIs) under finite Gamma mixture models. The class includes the Gini coefficient, the $m$th Gini index, the Extended $m$th Gini index, linear order-statistic inequality indices, and several entropy- and dispersion-based measures. Building on the Gamma--Dirichlet representation, we derive exact expressions for the expectation and finite-sample bias of the natural estimator of an NSII based on a $U$-statistic when the population follows a finite Gamma mixture with a common rate parameter. The resulting bias formula explicitly characterizes the effects of the mixture weights and component shape parameters. We also establish that the bias converges to zero as the sample size increases, yielding asymptotic unbiasedness. Furthermore, we derive exact finite-sample expressions for the variance and mean squared error of the estimator by accounting for the overlap structure among the subsets entering the $U$-statistic. In addition, we establish strong consistency and asymptotic normality. When the mixture degenerates to a single Gamma distribution, the finite-sample bias vanishes, recovering the exact unbiasedness result previously obtained for NSIIs under the Gamma model. Numerical illustrations and simulation studies are provided to investigate the finite-sample behavior of the estimator and to illustrate the theoretical findings. Overall, the proposed framework provides a unified characterization of finite-sample estimation under heterogeneous Gamma populations and quantifies the effects of mixture heterogeneity on bias, variance, and mean squared error.
	\end{abstract}
%


\section{Introduction}

Measures of inequality, variability, and concentration play an important role in economics, finance, demography, and other fields in which the distribution of a positive-valued variable is of interest. Among the most widely used measures is the Gini coefficient, originally introduced by \cite{Gini1912,Gini1936}, which quantifies the degree of inequality in a population. Several extensions and alternative measures have subsequently been proposed to capture different aspects of heterogeneity, dispersion, and concentration. These include the $m$th Gini index of \cite{Gavilan-Ruiz2024} and \cite{Vila2025}, entropy-based measures \citep{Shannon1948}, the Atkinson index \citep{Atkinson1970}, and several order-statistic-based inequality measures. Despite their different interpretations, many of these quantities can be expressed through expectations of symmetric homogeneous functions of independent observations.

A general framework encompassing a broad collection of such measures was developed by \cite{Vila2026NSII}. Let $g:(0,\infty)^m\rightarrow[0,\infty)$ be a measurable symmetric function that is homogeneous of degree one and satisfies an appropriate normalization condition. The associated normalized scale-invariant inequality index (NSII) is defined by
\[
I(g)
=
\frac{\mathbb{E}\left[g(X_1,\ldots,X_m)\right]}
     {m C_g\mu},
\]
where $\mu=\mathbb{E}[X]$ and $C_g$ is the sharp normalizing constant. This formulation includes, among others, the Gini coefficient, the $m$th Gini index, the extended $m$th Gini index \citep{Vila2026a}, linear order-statistic inequality indices \citep{Vila2026b}, power-type indices \citep{Atkinson1970}, Shannon entropy \citep{Shannon1948,CoverThomas2006}, product-type indices \citep{Bullen2003}, symmetric subcase indices \citep{MarshallOlkinArnold2011}, and the maximum and minimum indices \citep{DavidNagaraja2003}. Thus, rather than studying the finite-sample properties of each measure separately, the NSII framework permits a unified analysis based on the properties of the generating function $g$.

An important question concerns the statistical estimation of NSIIs. \cite{Vila2026NSII} established a general unbiased estimation result when the underlying population follows a Gamma distribution. Their approach exploits the Gamma--Dirichlet decomposition, under which the normalized vector of Gamma observations is independent of their total sum. This representation allows both the population index and its natural estimator based on a $U$-statistic to be expressed in terms of expectations with respect to Dirichlet distributions. In particular, the estimator is exactly unbiased for every NSII satisfying the stated regularity conditions and for every sample size $n\geqslant m$. This result provides a unified explanation for previously established unbiasedness properties of several particular inequality measures under Gamma populations and shows that exact unbiasedness is a structural consequence of the Gamma--Dirichlet decomposition rather than a property specific to the Gini coefficient or to a particular choice of $g$.

The assumption of a single Gamma population, however, imposes a degree of homogeneity that may be inappropriate when the population consists of subpopulations with different distributional characteristics. In income analysis, for example, populations may contain groups with substantially different income levels, and finite mixtures of parametric distributions provide a natural way to represent such heterogeneity. Gamma mixtures have therefore been used in the modeling of income distributions and related economic applications; see, for example, \cite{Chotikapanich2008}. More generally, mixture models provide a flexible representation in which observations arise from several latent components, each characterized by its own shape parameter and associated with a mixing probability.

Introducing population heterogeneity changes the finite-sample estimation problem substantially. In particular, the exact unbiasedness obtained under a single Gamma distribution does not generally extend to a finite Gamma mixture. \cite{VilaSaulo2025GiniMixture} studied this phenomenon for the Gini coefficient and showed that its estimator can exhibit finite-sample bias under Gamma mixture populations. More recently, \cite{mth-mixtures} considered the $m$th Gini index under finite Gamma mixtures and derived explicit expressions for its population value and for the expectation and bias of its natural estimator based on a $U$-statistic, together with asymptotic results under a common rate parameter. These contributions demonstrate that the finite-sample behavior of inequality estimators can change substantially in the presence of mixture heterogeneity.

The available results for Gamma mixtures, however, remain specialized to particular inequality measures. In particular, the analysis of \cite{VilaSaulo2025GiniMixture} concerns the classical Gini coefficient, whereas \cite{mth-mixtures} focuses on the $m$th Gini index. A general finite-sample theory for the estimator associated with an arbitrary NSII under a finite Gamma mixture is still missing. In particular, unified expressions for its finite-sample expectation, bias, variance, and mean squared error have not been established. Thus, while the single-Gamma framework provides exact unbiasedness for the entire NSII class, the corresponding finite-sample theory under Gamma mixtures remains to be developed.

The purpose of the present paper is to develop such a finite-sample framework for NSII estimation under finite Gamma mixture models. We consider a finite mixture of Gamma distributions sharing a common rate parameter, while allowing the mixture components to have different shape parameters and mixing probabilities. This assumption preserves the Gamma--Dirichlet structure conditionally on the latent mixture components and makes it possible to derive exact representations for the relevant quantities. Specifically, we derive general expressions for the expectation and finite-sample bias of the natural estimator based on a $U$-statistic for the entire NSII class. The resulting bias formula explicitly describes the effects of the mixture weights and component shape parameters and provides a basis for investigating the sign and magnitude of the bias. We further establish that the bias converges to zero as the sample size increases, thereby proving asymptotic unbiasedness. When the mixture degenerates to a single Gamma component, the finite-sample bias vanishes identically, recovering the unbiasedness result of \cite{Vila2026NSII}.

We then obtain exact finite-sample expressions for the variance and mean squared error of the estimator by accounting for the dependence between overlapping subsamples in the $U$-statistic. These expressions depend on the mixture weights and component shape parameters, whereas the common rate parameter disappears after scale-invariant normalization. In addition, we establish strong consistency and asymptotic normality, providing a comprehensive account of both finite-sample and large-sample estimation properties under heterogeneous Gamma populations.

The numerical illustrations and simulation study complement the theoretical results by examining the finite-sample behavior of the estimator under different Gamma mixture configurations. In particular, they investigate how the mixture weights, component shape parameters, and sample size affect the bias, variance, and mean squared error. The numerical results also illustrate the asymptotic decay of the bias and the consistency of the estimator as the sample size increases.

The remainder of the paper is organized as follows. Section~\ref{Exact expectation and bias of NSIIs under Gamma mixture models} develops the main theoretical results for NSII estimation under finite Gamma mixture models, including the population representation, exact expectation and finite-sample bias, consistency, asymptotic normality, and the asymptotic behavior of the bias. Section~\ref{sec:variance} derives exact finite-sample expressions for the variance and mean squared error of the estimator. Section~\ref{Simulation study} presents numerical illustrations and simulation results to assess its finite-sample behavior and illustrate the theoretical findings. Finally, Section~\ref{sect-5} presents the concluding remarks.

\section{
NSII estimation under Gamma mixture models
}
\label{Exact expectation and bias of NSIIs under Gamma mixture models}

This section derives exact representations for a broad class of normalized scale-invariant inequality indices under finite Gamma mixture models. We obtain closed-form expressions for the expectation and finite-sample bias of their corresponding $U$-statistic estimators, establish strong consistency and asymptotic normality, and conclude by examining the non-negativity and asymptotic behavior of the bias.

\subsection{Gamma mixture model and NSII framework}
\label{sec:gamma-mixture-nsii}

Throughout this section, let $X$ be a random variable following a
Gamma mixture distribution, denoted by
$
X\sim\mathrm{GM}(\boldsymbol{\theta}),
\ 
\boldsymbol{\theta}
=
(\pi_1,\ldots,\pi_\ell,\alpha_1,\ldots,\alpha_\ell,\lambda)^\top,
$
where
$
\pi_r>0,\ 
\sum_{r=1}^{\ell}\pi_r=1,\ 
\alpha_r>0,\ 
\lambda>0.
$
The probability density function of $X$ is
\begin{equation}
\label{eq:gamma-mixture}
f_X(x;\boldsymbol{\theta})
=
\sum_{r=1}^{\ell}
\pi_r \, 
\frac{\lambda^{\alpha_r}}{\Gamma(\alpha_r)} \, 
x^{\alpha_r-1}\text{e}^{-\lambda x},
\quad x>0,
\end{equation}
where $\Gamma(\cdot)$ denotes the Gamma function. Its mean is
$
\mu\coloneqq\mathbb{E}[X]
=
({1}/{\lambda})
\sum_{r=1}^{\ell}\pi_r\alpha_r.
$
Figure 1 in \cite{VilaSaulo2025GiniMixture} illustrates the flexibility
of this class of distributions and some of its important special cases.


Let $X_1,\ldots,X_n$ be independent and identically distributed (iid)
copies of $X\sim\mathrm{GM}(\boldsymbol{\theta})$. Then, for each
$i=1,\ldots,n$,
\begin{align*}
\label{representation-X-i}
X_i
\stackrel{d}{=}
\sum_{r=1}^{\ell}
\mathds{1}_{\{Y_i=r\}}Z_{ir},
\end{align*}
where $\stackrel{d}{=}$ denotes equality in distribution,  $\mathds{1}_A$ is the indicator function of the event $A$, $Y_1,\ldots,Y_n$ are iid discrete random variables with probability mass functions
$
\mathbb{P}(Y_i=r)=\pi_r,
\ r=1,\ldots,\ell,
$
and, 
\[
X_i\mid \{Y_i=r\}
\overset{d}{=}
Z_{ir}\sim\operatorname{Gamma}(\alpha_r,\lambda).
\]

For $m\geqslant 2$, let
$g:(0,\infty)^m\to[0,\infty)$ be a measurable symmetric function,
satisfying
$
g(x_1,\ldots,x_m)
=
g(x_{\pi(1)},\ldots,x_{\pi(m)})
$
for every permutation $\pi$ of $\{1,\ldots,m\}$, and the following
conditions:
\begin{enumerate}
    \item[(A1)]
    \[
    g(c x_1,\ldots,c x_m)
    =
    c g(x_1,\ldots,x_m),
    \quad c>0,
    \]
    so that $g$ is homogeneous of degree one;

    \item[(A2)]
    \[
    C_g:=
    \sup_{\boldsymbol{x}\in(0,\infty)^m}
    \frac{g(x_1,\ldots,x_m)}
    {\sum_{i=1}^m x_i}
    \in(0,\infty).
    \]
\end{enumerate}

Following \cite{Vila2026NSII}, the normalized scale-invariant inequality index (NSII) generated by $g$ is defined as
\begin{equation}
\label{eq:index}
I(g)
\coloneqq
I(g;X)
=
\frac{\mathbb{E}[g(X_1,\ldots,X_m)]}
{mC_g\mu}.
\end{equation}
Since (A2) implies
$
0\leqslant g(x_1,\ldots,x_m)
\leqslant C_g\sum_{i=1}^m x_i,
$
it follows that
$
0\leqslant I(g)\leqslant 1.
$
Moreover, if $g(x,\ldots,x)=0$ for all $x>0$ and $X$ is degenerate,
then $I(g)=0$. By (A1),
$
I(g;cX)=I(g;X),\ c>0,
$
so that $I(g)$ is scale-invariant.

Since the distribution of $X$ is fixed throughout this section, we
write $I(g)$ without explicitly indicating its dependence on the
distribution.

The main examples of NSIIs, together with their generating functions,
sharp normalizing constants, and corresponding index forms, are
summarized in Table~\ref{NPRI-index}. This collection of examples
first appeared in \cite{Vila2026NSII}.
\begin{table}[htb!] 
	\centering
	\resizebox{\textwidth}{!}{
		\renewcommand{\arraystretch}{2.2}
		\setlength{\extrarowheight}{2.2pt}
		\begin{tabular}{|c|c|c|c|}
			\hline
			\textbf{Name} & ${g}(x_1,\ldots,x_m)$ & $C_{ g}$ & ${I}(g)$\\
			\hline
            
			Gini {coefficient} 
			& $|x_1-x_2|$
			& $1$ 
			& $\displaystyle \frac{\mathbb{E}|X_1-X_2|}{2\mu}$ 
			\\
			\hline
			
			The $m$th Gini index 
			& $x_{m:m}-x_{1:m}$
			& $1$ 
			& $\displaystyle \frac{\mathbb{E}[X_{m:m}-X_{1:m}]}{m\mu}$
			\\
			\hline
			
			Extended $m$th Gini index 
			& $x_{k:m}-x_{j:m}$, $1\leqslant j<k\leqslant m$
			& {$\displaystyle \frac{1}{m-k+1}$ }
   			& 
            $\displaystyle \frac{
            {(m-k+1)}
            \mathbb{E}[X_{k:m}-X_{j:m}]}{m\mu}$         
			\\
			\hline
			
			Linear order-statistic (OS) inequality index 
			& \makecell{$\displaystyle \sum_{k=1}^m a_k x_{k:m}$, 
			\vspace*{0.1cm}$\displaystyle \sum_{k=1}^m a_k = 0$, \\  \vspace*{0.1cm}  $a_1 \leqslant \cdots \leqslant 
            a_m,
            \ 
            {a_m =1}
            $
            }
			& $1$ 
			& $\displaystyle
			\frac{\sum_{k=1}^m a_k\mathbb{E}[X_{k:m}]}{m\mu}$ 
			\\ 
			\hline
			
			Smoothed CV index (SCV)
			& $\displaystyle
			\frac{(x_1-x_2)^2}{x_1+x_2}$
			& $1$ 
			& $ \displaystyle
			\frac{\mathbb{E}\big[\frac{(X_1-X_2)^2}{X_1+X_2}\big]}{2\mu}$ 
			\\
			\hline
			
			Power index ($p>0$) 
			& $\displaystyle
			x_1^p(x_1+x_2)^{1-p}+x_2^p(x_1+x_2)^{1-p}$
			& $\max\{1,2^{1-p}\}$ 
			& $\displaystyle
			\frac{\mathbb{E}[g(X_1,X_2)]}{2\mu \max\{1,2^{1-p}\}}$
			\\
			\hline
			
			Shannon entropy index 
			& $\displaystyle -x_1\log\left(\frac{x_1}{x_1+x_2}\right)-x_2\log\left(\frac{x_2}{x_1+x_2}\right)$ 
			& $\log(2)$ 
			& $\displaystyle \frac{\mathbb{E}[g(X_1,X_2)]}{2\mu\log(2)}$ 
			\\
			\hline
			
			Product index 
			& $\displaystyle\frac{x_1x_2}{x_1+x_2}$
			& $\displaystyle\frac{1}{4}$ 
			& $\displaystyle \frac{2\,\mathbb{E}\!\left[\frac{X_1X_2}{X_1+X_2}\right]}{\mu}$ 
			\\
			\hline
			
			{Symmetric subcase index}
			& $\displaystyle \frac{x_1^\theta x_2^\beta}{(x_1+x_2)^{\theta+\beta-1}}$, $\theta,\beta>0$, {\color{black}$\theta=\beta$}
			& $ \displaystyle \frac{\theta^\theta\beta^\beta}{(\theta+\beta)^{\theta+\beta}}$ 
			& $ \displaystyle \frac{(\theta+\beta)^{\theta+\beta}\mathbb{E}[g(X_1,X_2)]}{
            2\theta^\theta\beta^\beta
            \mu}$ 
			\\
			\hline
			
			
			Maximum index  
			& {$\max\{x_1,x_2\}$}
			& $1$ 
			& $\displaystyle\frac{\mathbb{E}[\max\{X_1,X_2\}]}{2\mu}$ 
			\\
			\hline
			
			Minimum index 
			& $\displaystyle\min\{x_1,x_2\}$
			& $\displaystyle\tfrac{1}{2}$ 
			& $\displaystyle \frac{\mathbb{E}[\min\{X_1,X_2\}]}{\mu}$ 
			\\
			\hline
			
			
		\end{tabular}
	}
\caption{
	Examples of {NSII}s. The examples listed in this table are related to the Gini index \citep{Gini1912,Gini1936,Yitzhaki2013}, the $m$th Gini index	\citep{Gavilan-Ruiz2024,Vila2025}, the Extended $m$th Gini index	\citep{Vila2026a}, the Linear order-statistic inequality index	\citep{Vila2026b}, the Power index \citep{Atkinson1970},	the Shannon entropy index \citep{Shannon1948,CoverThomas2006}, the Product index \citep{Bullen2003}, the Symmetric subcase (of the asymmetric) index	\citep{MarshallOlkinArnold2011}, and the Maximum and Minimum indices	\citep{DavidNagaraja2003}. Reproduced from \cite{Vila2026NSII}.}
	\label{NPRI-index}
\end{table}

Define the $U$-statistic
\begin{align}\label{eq:Us}
U_n(g)
\coloneqq
{1\over \binom{n}{m}}
\sum_{1\leqslant i_1<\cdots<i_m\leqslant n}
g(X_{i_1},\ldots,X_{i_m}).
\end{align}
Then the natural estimator based on a $U$-statistic
 of $I(g)$ is
\begin{align}\label{eq:Tn}
\widehat{I(g)}
=
\frac{U_n(g)}{mC_g\overline X},
\quad n\geqslant m,
\end{align}
where $\overline{X}=(1/{n})\sum_{i=1}^n X_i$ denotes the sample mean.

The derivation of the subsequent results relies on the following structural representation.


\begin{lemma}[Gamma--Dirichlet mixture representation]
\label{lem:GD}
Let
$
Y_i
\in
\{1,\ldots,\ell\},
$
be latent variables satisfying
$
\mathbb{P}(Y_i=r)
=
\pi_r,
\ 
r=1,\ldots,\ell.
$
Let
\[
T\coloneqq\sum_{i=1}^m X_i,
\quad
V_i^*\coloneqq\frac{X_i}{T},
\quad i=1,\ldots,m.
\]
Then, 
conditional on
$
Y_1={j_1},
\ldots,
Y_m={j_m},
$
the random vector admits the stochastic representation
\begin{align}\label{rep-est}
(X_1,\ldots,X_m)^\top
\stackrel{d}{=}
T(V_1^*,\ldots,V_m^*)^\top,
\end{align}
where
\begin{enumerate}
    \item[(a)]  $
T\mid 
\left\{
Y_1=j_1,\ldots,Y_m=j_m
\right\}
\sim
\text{Gamma}
\left(
\sum_{i=1}^{m}\alpha_{j_i},
\lambda
\right),
$
    \item[(b)] 
$
(V_1^*,\ldots,V_m^*)^\top\mid 
\left\{
Y_1=j_1,\ldots,Y_m=j_m
\right\}
\sim
\operatorname{Dir}
(
\alpha_{j_1},
\ldots,
\alpha_{j_m}
)
$, and
    \item[(c)] 
\(T\) and
\((V_1^*,\ldots,V_m^*)^\top\)
being conditionally independent given
$Y_1=j_1,\ldots,Y_m=j_m$.
\end{enumerate}

Moreover, if
$
S
=
\sum_{i=1}^{n}X_i,
$
then
\[
\left(
\frac{X_1}{S},
\ldots,
\frac{X_n}{S}
\right)^\top
\Bigg|
\left\{
Y_1=j_1,\ldots,Y_n=j_n
\right\}
\sim
\operatorname{Dir}
(
\alpha_{j_1},
\ldots,
\alpha_{j_n}
)
\]
and 
\(S\) and
\((X_1/S,\ldots,X_n/S)^\top\)
are conditionally independent given
$Y_1=j_1,\ldots,Y_n=j_n$.
\end{lemma}
\begin{proof}
Conditional on
$
Y_i=j_i,\  i=1,\ldots,m,
$
the random variables $X_1,\ldots,X_m$ are independent, with
$
X_i\mid \{Y_i=j_i\}
\sim
\operatorname{Gamma}(\alpha_{j_i},\lambda),
\ i=1,\ldots,m.
$
The classical Gamma--Dirichlet decomposition \citep{Devroye1986}
then implies that
\[
T
\mid 
\left\{Y_1=j_1,\ldots,Y_m=j_m\right\}
\sim
\operatorname{Gamma}
\left(
\sum_{i=1}^{m}\alpha_{j_i},\lambda
\right),
\]
and
\[
(V_1^*,\ldots,V_m^*)^\top
\mid 
\left\{Y_1=j_1,\ldots,Y_m=j_m\right\}
\sim
\operatorname{Dir}
(\alpha_{j_1},\ldots,\alpha_{j_m}),
\]
with $T$ and $(V_1^*,\ldots,V_m^*)^\top$ being conditionally independent given
$Y_1=j_1,\ldots,Y_m=j_m$.

Consequently, using the identity
\[
(X_1,\ldots,X_m)^\top
=
T
(V_1^*,\ldots,V_m^*)^\top,
\]
we obtain the
stochastic representation in \eqref{rep-est}.

Applying the same argument to $X_1,\ldots,X_n$, conditional on
$Y_1=j_1,\ldots,Y_n=j_n$,
and letting $S=\sum_{i=1}^{n}X_i$, we obtain
\[
S
\;\big|\;
\left\{Y_1=j_1,\ldots,Y_n=j_n\right\}
\sim
\operatorname{Gamma}
\left(
\sum_{i=1}^{n}\alpha_{j_i},\lambda
\right),
\]
and
\[
\left(
\frac{X_1}{S},\ldots,\frac{X_n}{S}
\right)^\top
\Bigg|\,
\left\{Y_1=j_1,\ldots,Y_n=j_n\right\}
\sim
\operatorname{Dir}
(\alpha_{j_1},\ldots,\alpha_{j_n}),
\]
with $S$ and the normalized vector $\left(
{X_1}/{S},\ldots,{X_n}/{S}
\right)^\top$ being conditionally independent  given
$Y_1=j_1,\ldots,Y_n=j_n$.
\end{proof}


Lemma~\ref{lem:GD} constitutes the key probabilistic ingredient of this section. The representation of the population inequality index, the expectation of its estimator, and its exact finite-sample bias are all direct consequences of this decomposition together with the homogeneity property of the kernel \(g\).

\subsection{Population representation}

The following proposition provides an explicit representation of the population inequality index in terms of Dirichlet expectations.

\begin{proposition}
\label{prop:Ig}
Let $g:(0,\infty)^m\to[0,\infty)$ be a measurable symmetric function
satisfying {\rm(A1)}--{\rm(A2)}. Then the NSII $I(g)$ admits the
representation
\begin{equation}
\label{eq:Ig-Dir}
I(g)
=
\frac{1}
{mC_g\sum_{r=1}^{\ell}\pi_r\alpha_r}
\sum_{j_1,\ldots,j_m=1}^{\ell}
\left(
\prod_{i=1}^{m}\pi_{j_i}
\right)
\left(
\sum_{i=1}^{m}\alpha_{j_i}
\right)
\mathbb{E}_{\operatorname{Dir}
(\alpha_{j_1},\ldots,\alpha_{j_m})}
\left[
g(V_1,\ldots,V_m)
\right],
\end{equation}
where
$
(V_1,\ldots,V_m)^\top
\sim
\operatorname{Dir}
(\alpha_{j_1},\ldots,\alpha_{j_m}).
$
\end{proposition}
\begin{proof}
By the law of total expectation,
\[
\mathbb{E}[g(X_1,\ldots,X_m)]
=
\mathbb{E}
\left[
\mathbb{E}
\left(
g(X_1,\ldots,X_m)
\mid
Y_1,\ldots,Y_m
\right)
\right].
\]

Condition on
$
Y_1={j_1},
\ldots,
Y_m={j_m}.
$
By Lemma~\ref{lem:GD},
$
(X_1,\ldots,X_m)^\top
\stackrel{d}{=}
T(V_1^*,\ldots,V_m^*)^\top,
$
where
$$
T\mid 
\left\{
Y_1=j_1,\ldots,Y_m=j_m
\right\}
\sim
\text{Gamma}
\left(
\sum_{i=1}^{m}\alpha_{j_i},
\lambda
\right),
$$
$$
(V_1^*,\ldots,V_m^*)^\top\mid 
\left\{
Y_1=j_1,\ldots,Y_m=j_m
\right\}
\sim
\operatorname{Dir}
(
\alpha_{j_1},
\ldots,
\alpha_{j_m}
),
$$
and \(T\) is conditionally independent of
\((V_1^*,\ldots,V_m^*)^\top\).

Since \(g\) is positively homogeneous of degree one,
$
g(TV_1^*,\ldots,TV_m^*)
=
Tg(V_1^*,\ldots,V_m^*),
$
and therefore, for $(V_1,\ldots,V_m)^\top
\sim \operatorname{Dir}
(\alpha_{j_1},\ldots,\alpha_{j_m})
$,
\begin{align*}
\mathbb{E}
\bigl[
g(X_1,\ldots,X_m)
&\mid
Y_1=j_1,\ldots,Y_m=j_m
\bigr]
\\[0,2cm]
&=
\mathbb{E}[T\mid
Y_1=j_1,\ldots,Y_m=j_m]
\,
\mathbb{E}
[g(V_1^*,\ldots,V_m^*)\mid
Y_1=j_1,\ldots,Y_m=j_m]
\\[0,2cm]
&=
\frac{
\sum_{i=1}^{m}\alpha_{j_i}
}
{\lambda}
\,
\mathbb{E}_{\operatorname{Dir}
(\alpha_{j_1},\ldots,\alpha_{j_m})}
[g(V_1,\ldots,V_m)],
\end{align*}
where, in the last equality, we used
$
\mathbb{E}[T\mid Y_1=j_1,\ldots,Y_m=j_m]
=
\left(
\sum_{i=1}^{m}\alpha_{j_i}
\right)/
\lambda.
$
Thus,
\[
\begin{aligned}
\mathbb{E}
\left[
g(X_1,\ldots,X_m)
\mid
Y_1,\ldots,Y_m
\right]
=
\frac{
\sum_{i=1}^{m}\alpha_{Y_i}
}
{\lambda}
\,
\mathbb{E}_{\operatorname{Dir}
(\alpha_{Y_1},\ldots,\alpha_{Y_m})}
[g(V_1,\ldots,V_m)].
\end{aligned}
\]

Taking expectation with respect to the latent variables $Y_1,\ldots,Y_n$ and applying the law of total expectation yields
\[
\begin{aligned}
\mathbb{E}[g(X_1,\ldots,X_m)]
=
\frac1{\lambda}
\sum_{j_1,\ldots,j_m=1}^{\ell}
\left(
\prod_{i=1}^{m}\pi_{j_i}
\right)
\left(
\sum_{i=1}^{m}\alpha_{j_i}
\right)
\mathbb{E}_{\operatorname{Dir}
(\alpha_{j_1},\ldots,\alpha_{j_m})}
[g(V_1,\ldots,V_m)].
\end{aligned}
\]

Finally, since
$
\mu
=
({1}/{\lambda})\sum_{r=1}^{\ell}\pi_r\alpha_r,
$
the stated representation follows immediately from the above equality
and the definition of $I(g)$.
\end{proof}

The previous proposition shows that the population inequality index depends on the Gamma mixture model only through the mixture weights and shape parameters. In particular, the common rate parameter \(\lambda\) disappears after normalization by the population mean, reflecting the scale invariance of \(I(g)\). This representation also reveals that evaluating the index reduces to computing expectations with respect to Dirichlet distributions, a property that will play a central role in the derivation of the expectation and bias of the estimator.

For the different choices of the generator $g$, the expectations
$
\mathbb{E}_{\operatorname{Dir}(\alpha_{j_1},\ldots,\alpha_{j_m})}
\left[
g(V_1,\ldots,V_m)
\right],
$
where
$
(V_1,\ldots,V_m)^\top
\sim
\operatorname{Dir}(\alpha_{j_1},\ldots,\alpha_{j_m}),
$
appearing in Proposition~\ref{prop:Ig} are explicitly given in
Table~\ref{tab:dirichlet_expectations}.

\subsection{Expectation and finite-sample bias}

We now derive an exact expression for the expectation of the estimator \(\widehat{I(g)}\) in \eqref{eq:Tn}. The result shows that its expectation can also be written entirely in terms of Dirichlet expectations.

\begin{theorem}
\label{thm:Expectation}
Let $g:(0,\infty)^m\to[0,\infty)$ be a measurable symmetric function
satisfying {\rm(A1)}--{\rm(A2)}.  Then the expectation of the estimator \(\widehat{I(g)}\) can be written as
\[
\begin{aligned}
\mathbb{E}[\widehat{I(g)}]
=
\frac{n}{mC_g}
\sum_{j_1,\ldots,j_n=1}^{\ell}
\left(
\prod_{i=1}^{n}\pi_{j_i}
\right) 
\left(
{\sum_{i=1}^m \alpha_{j_i}} 
\right)
{
\mathbb{E}_{\operatorname{Dir}
(\alpha_{j_1},\ldots,\alpha_{j_m})}
\!\left[
g(V_1,\ldots,V_m)
\right]
\over 
\sum_{i=1}^n \alpha_{j_i}
},
\end{aligned}
\]
where
$
(V_1,\ldots,V_m)^\top
\sim
\operatorname{Dir}
(
\alpha_{j_1},
\ldots,
\alpha_{j_m}
)
$.
\end{theorem}
\begin{proof}
By the definition of $\widehat{I(g)}$ and the exchangeability of
$X_1,\ldots,X_n$,
\[
\mathbb{E}[\widehat{I(g)}]
=
\frac1{mC_g}\, 
\mathbb{E}
\left[
\frac{
g(X_1,\ldots,X_m)
}
{
\overline X
}
\right].
\]

Using
$
\overline X=S/n,
\
S=\sum_{i=1}^{n}X_i,
$
we obtain
\[
\mathbb{E}[\widehat{I(g)}]
=
\frac n{mC_g} 
\,
\mathbb{E}
\left[
\frac{
g(X_1,\ldots,X_m)
}
{
S
}
\right].
\]

By assumption (A1),
$
g(X_1,\ldots,X_m)
=
S
g\left(
{X_1}/{S},
\ldots,
{X_m}/{S}
\right),
$
which immediately gives
\[
\mathbb{E}[\widehat{I(g)}]
=
\frac n{mC_g}
\,
\mathbb{E}
\left[
g
\left(
\frac{X_1}{S},
\ldots,
\frac{X_m}{S}
\right)
\right].
\]

Now,
Lemma~\ref{lem:GD} implies
$$
\left(
\frac{X_1}{S},
\ldots,
\frac{X_n}{S}
\right)^\top 
\Bigg|\,
\{
Y_1=j_1,\ldots,Y_n=j_n
\}
\sim
\operatorname{Dir}
(
\alpha_{j_1},
\ldots,
\alpha_{j_n}
),
$$
and therefore, for $(U_1,\ldots,U_n)^\top
\sim
\operatorname{Dir}
(
\alpha_{j_1},
\ldots,
\alpha_{j_n}
)$,
\begin{align*}
\mathbb{E}
\left[
g
\left(
\frac{X_1}{S},
\ldots,
\frac{X_m}{S}
\right)
\middle|
Y_1=j_1,\ldots,Y_n=j_n
\right]
&=
\mathbb{E}_{\operatorname{Dir}
(\alpha_{j_1},\ldots,\alpha_{j_n})}
\!\left[
g(U_1,\ldots,U_m)
\right],
\\[0,2cm]
&=
{\sum_{i=1}^m \alpha_{j_i}\over \sum_{i=1}^n \alpha_{j_i}} 
\,
\mathbb{E}_{\operatorname{Dir}
(\alpha_{j_1},\ldots,\alpha_{j_m})}
\!\left[
g(V_1,\ldots,V_m)
\right],
\end{align*}
where the last equality follows from the aggregation property of the
Dirichlet distribution. Specifically, with
$$
(V_1,\ldots,V_m)^\top 
\stackrel{d}{=}
\left(\frac{U_1}{\sum_{i=1}^m U_i},\ldots,\frac{U_m}{\sum_{i=1}^m U_i}\right)^\top
\sim
\operatorname{Dir}(\alpha_{j_1},\ldots,\alpha_{j_m}),
$$
independently of $\sum_{i=1}^m U_i$, and
$
\mathbb{E}[\sum_{i=1}^m U_i]
=
{\sum_{i=1}^m\alpha_{j_i}}/
{\sum_{i=1}^n\alpha_{j_i}},
$
while $g$ is homogeneous of degree one.

Applying the law of total expectation yields the desired identity, completing the proof.
\end{proof}

Theorem~\ref{thm:Expectation} reveals that the expectation of the estimator depends exclusively on Dirichlet expectations. Consequently, the influence of the Gamma mixture model is completely characterized by the mixture weights and the corresponding shape parameters, while the common rate parameter plays no role. This property mirrors the population representation established in Proposition~\ref{prop:Ig} and makes it possible to derive an exact expression for the finite-sample bias by direct comparison of the two formulas.

The expectations
$
\mathbb{E}_{\operatorname{Dir}(\alpha_{j_1},\ldots,\alpha_{j_m})}
\left[
g(V_1,\ldots,V_m)
\right],
$
where
$
(V_1,\ldots,V_m)^\top
\sim
\operatorname{Dir}(\alpha_{j_1},\ldots,\alpha_{j_m}),
$
required in Theorem~\ref{thm:Expectation} are given explicitly in
Table~\ref{tab:dirichlet_expectations} for the different choices of $g$.

\begin{table}[htb!]
	\centering
	\resizebox{\textwidth}{!}{
		\renewcommand{\arraystretch}{2.2}
		\setlength{\extrarowheight}{2.2pt}
		\begin{tabular}{|c|c|}
			\hline
			\textbf{Name} &
			$\displaystyle
			\mathbb{E}_{\operatorname{Dir}(\alpha_{j_1},\ldots,\alpha_{j_m})}
			[g(V_1,\ldots,V_m)]$
			\\
			\hline

			Gini coefficient
			&
			{$\displaystyle
			1-
			\frac{
				2\left\{B_{1/2}(\alpha_{j_1}+1,\alpha_{j_2})
				+
				B_{1/2}(\alpha_{j_2}+1,\alpha_{j_1})\right\}
			}{
				B(\alpha_{j_1},\alpha_{j_2})
			}$}
			\\
			\hline

			$m$th Gini index
			&
			$\displaystyle
			\frac{
				\Gamma\!\left(\sum_{\ell=1}^{m}\alpha_{j_\ell}\right)
			}{
				\prod_{i=1}^{m}\Gamma(\alpha_{j_i})
			}
			\int_{\Delta_{m-1}}
			(v_{m:m}-v_{1:m})
			\prod_{i=1}^{m}v_i^{\alpha_{j_i}-1}
			\,{\rm d}\boldsymbol{v}$
			\\
			\hline

			Extended $m$th Gini
			&
			$\displaystyle
			\frac{
				\Gamma\!\left(\sum_{\ell=1}^{m}\alpha_{j_\ell}\right)
			}{
				\prod_{i=1}^{m}\Gamma(\alpha_{j_i})
			}
			\int_{\Delta_{m-1}}
			(v_{k:m}-v_{j:m})
			\prod_{i=1}^{m}v_i^{\alpha_{j_i}-1}
			\,{\rm d}\boldsymbol{v}$
			\\
			\hline

			Linear order-statistic (OS) inequality index
			&
			$\displaystyle
			\frac{
				\Gamma\!\left(\sum_{\ell=1}^{m}\alpha_{j_\ell}\right)
			}{
				\prod_{i=1}^{m}\Gamma(\alpha_{j_i})
			}
			\int_{\Delta_{m-1}}
			\left(\sum_{k=1}^{m}a_kv_{k:m}\right)
			\prod_{i=1}^{m}v_i^{\alpha_{j_i}-1}
			\,{\rm d}\boldsymbol{v}$
			\\
			\hline

			Smoothed CV (SCV)
			&
			$\displaystyle
			\frac{
				(\alpha_{j_1}-\alpha_{j_2})^2
				+\alpha_{j_1}+\alpha_{j_2}
			}{
				(\alpha_{j_1}+\alpha_{j_2})
				(\alpha_{j_1}+\alpha_{j_2}+1)
			}$
			\\
			\hline

			Power index $p>0$
			&
			$\displaystyle
			\frac{
				B(\alpha_{j_1}+p,\alpha_{j_2})
				+
				B(\alpha_{j_1},\alpha_{j_2}+p)
			}{
				B(\alpha_{j_1},\alpha_{j_2})
			}$
			\\
			\hline

			Shannon entropy
			&
			$\displaystyle
			\frac{
				\alpha_{j_1}
				\left\{
					\psi(\alpha_{j_1}+\alpha_{j_2}+1)
					-\psi(\alpha_{j_1}+1)
				\right\}
				+
				\alpha_{j_2}
				\left\{
					\psi(\alpha_{j_1}+\alpha_{j_2}+1)
					-\psi(\alpha_{j_2}+1)
				\right\}
			}{
				\alpha_{j_1}+\alpha_{j_2}
			}$
			\\
			\hline

			Product
			&
			$\displaystyle
			\frac{
				\alpha_{j_1}\alpha_{j_2}
			}{
				(\alpha_{j_1}+\alpha_{j_2})
				(\alpha_{j_1}+\alpha_{j_2}+1)
			}$
			\\
			\hline

			Symmetric subcase index
			&
			$\displaystyle
			\frac{
				B(\alpha_{j_1}+\theta,\alpha_{j_2}+\beta)
			}{
				B(\alpha_{j_1},\alpha_{j_2})
			},
			\quad \theta,\beta>0,\ \theta=\beta$
			\\
			\hline

			Maximum
			&
			$\displaystyle
			\frac{
				B(\alpha_{j_1}+1,\alpha_{j_2})
				-
				B_{1/2}(\alpha_{j_1}+1,\alpha_{j_2})
				+
				B_{1/2}(\alpha_{j_1},\alpha_{j_2}+1)
			}{
				B(\alpha_{j_1},\alpha_{j_2})
			}$
			\\
			\hline

			Minimum
			&
			$\displaystyle
			\frac{
				B_{1/2}(\alpha_{j_1}+1,\alpha_{j_2})
				+
				B(\alpha_{j_1},\alpha_{j_2}+1)
				-
				B_{1/2}(\alpha_{j_1},\alpha_{j_2}+1)
			}{
				B(\alpha_{j_1},\alpha_{j_2})
			}$
			\\
			\hline
		\end{tabular}
	}

	\caption{Expressions for
	$\mathbb{E}_{\operatorname{Dir}(\alpha_{j_1},\ldots,\alpha_{j_m})}
	[g(V_1,\ldots,V_m)]$ for the inequality measures considered.
	Here,
	$B(a,b)=\Gamma(a)\Gamma(b)/\Gamma(a+b)$ is the beta function,
	$B_x(a,b)=\int_0^x t^{a-1}(1-t)^{b-1}\,{\rm d}t$ is the incomplete beta
	function, $\Gamma(a)$ is the Gamma function, and
	$\psi(a)=\Gamma'(a)/\Gamma(a)$ is the digamma function.
	Moreover,
	$\Delta_{m-1}=
	\{(v_1,\ldots,v_m)\in(0,1)^m:
	\sum_{i=1}^{m}v_i=1\}$.}
	\label{tab:dirichlet_expectations}
\end{table}

\begin{corollary}\label{cor:Bias}
Under the assumptions of Theorem~\ref{thm:Expectation}, we have
    \begin{multline*}
\operatorname{Bias}(\widehat{I(g)},I(g))
=
\frac{1}{mC_g}
\sum_{j_1,\ldots,j_m=1}^{\ell}
\left(
\prod_{i=1}^{m}\pi_{j_i}
\right) 
\left(
\sum_{i=1}^{m}\alpha_{j_i}
\right)
\mathbb{E}_{\operatorname{Dir}
(\alpha_{j_1},\ldots,\alpha_{j_m})}
\!\left[
g(V_1,\ldots,V_m)
\right]
\\[0,2cm]
\times 
\sum_{j_{m+1},\ldots,j_n=1}^{\ell}
\left(
\prod_{i=m+1}^{n}\pi_{j_i}
\right) 
\Bigg(
{n\over \sum_{i=1}^n \alpha_{j_i}} 
-
\frac1{\sum_{r=1}^{\ell}\pi_r\alpha_r}
\Bigg),
    \end{multline*}
where     
$
(V_1,\ldots,V_m)^\top 
\sim
\operatorname{Dir}
(
\alpha_{j_1},
\ldots,
\alpha_{j_m}
)
$.
\end{corollary}

{
\begin{remark}
The bias in Corollary \ref{cor:Bias} admits the following decomposition:
\begin{align*}
\operatorname{Bias}(\widehat{I(g)},I(g))
&=\frac{1}{mC_g}
\Bigg\{
\mathbb{E}\Bigg[
\left(\sum_{i=1}^{m}\alpha_{Y_i}\right)
\mathbb{E}_{\operatorname{Dir}(\alpha_{Y_1},\ldots,\alpha_{Y_m})}
\left[g(V_1,\ldots,V_m)\right]
\Bigg]
\left(
n\mathbb{E}\left[
\frac{1}{\sum_{i=1}^{n}\alpha_{Y_i}}
\right]
-\frac{1}{\sum_{r=1}^{\ell}\pi_r\alpha_r}
\right)
\\[0,2cm]
&+
n\operatorname{Cov}\Bigg(
\left(\sum_{i=1}^{m}\alpha_{Y_i}\right)
\mathbb{E}_{\operatorname{Dir}(\alpha_{Y_1},\ldots,\alpha_{Y_m})}
\left[g(V_1,\ldots,V_m)\right],
\frac{1}{\sum_{i=1}^{n}\alpha_{Y_i}}
\Bigg)
\Bigg\}.
\end{align*}
Since the function $x\mapsto 1/x$ is convex on $(0,\infty)$, Jensen's inequality gives
\[
\mathbb{E}\left[
\frac{1}{\sum_{i=1}^{n}\alpha_{Y_i}}
\right]
\geqslant
\frac{1}{\mathbb{E}\left[\sum_{i=1}^{n}\alpha_{Y_i}\right]}
=
\frac{1}{n\sum_{r=1}^{\ell}\pi_r\alpha_r}.
\]
Hence,
\[
n\mathbb{E}\left[
\frac{1}{\sum_{i=1}^{n}\alpha_{Y_i}}
\right]
-\frac{1}{\sum_{r=1}^{\ell}\pi_r\alpha_r}
\geqslant 0.
\]
Moreover, since $g\geqslant 0$, the expectation multiplying this nonnegative quantity is also nonnegative. Therefore, the first contribution to the bias is nonnegative. The sign of the bias as a whole, however, depends on the covariance term. In particular, a nonnegative covariance guarantees a nonnegative bias, whereas a negative covariance does not necessarily imply a negative bias: it must be sufficiently negative to offset the first contribution. Thus, the sign of the bias cannot, in general, be determined solely from the nonnegativity of $g$ or from Jensen's inequality.
\end{remark}
}


\begin{corollary}[Exact unbiasedness for Gamma populations]
\label{cor:gamma}
Suppose that the finite Gamma mixture reduces to a single Gamma distribution, namely,
$
\ell=1,
\
X\sim\text{Gamma}(\alpha,\lambda).
$
Then,
$
\mathbb{E}[\widehat{I(g)}]=I(g),
\ 
n\geqslant m,
$
or, equivalently,
\[
\operatorname{Bias}(\widehat{I(g)},I(g))=0.
\]
Hence, the U-statistic estimator is exactly unbiased for every NSII generated by a nonnegative measurable symmetric function satisfying {\rm(A1)}--{\rm(A2)}.
\end{corollary}
\begin{proof}
For $\ell=1$, we have $j_1=\cdots=j_n=1$, $\pi_1=1$ and $\alpha_1=\alpha$. Hence, 
\[
\begin{aligned}
{n\over \sum_{i=1}^n \alpha_{j_i}} 
-
\frac1{\sum_{r=1}^{\ell}\pi_r\alpha_r}
=
\frac{n}{n\alpha_1}
-
\frac{1}{\alpha_1}
=
0.
\end{aligned}
\]
From Corollary \ref{cor:Bias} it follows that $\operatorname{Bias}(\widehat{I(g)},I(g))=0$.

This completes the proof.
\end{proof}

The result in Corollary~\ref{cor:gamma} was first established in \cite{Vila2026NSII}. The corollary shows that Gamma distributions occupy a distinguished position within the class of finite Gamma mixtures. Although finite mixtures generally induce finite-sample bias, this bias vanishes identically when the mixing distribution degenerates to a single Gamma component. Consequently, the exact unbiasedness established in the Gamma case can be interpreted as a structural property arising from the Gamma--Dirichlet decomposition rather than as an isolated feature of a particular inequality measure.

\subsection{Strong consistency and asymptotic normality}

The representation of $\widehat{I(g)}$ as a ratio involving a $U$-statistic
allows us to establish its strong consistency and asymptotic normality under
the finite Gamma mixture model

\begin{theorem}\label{thm:strong-consistency}
Under the assumptions of Theorem~\ref{thm:Expectation},
$
\widehat{I(g)}
\xrightarrow{\mathrm{a.s.}}
I(g),
\ n\to\infty,
$
where $\xrightarrow{\mathrm{a.s.}}$ denotes almost sure convergence.
Thus, $\widehat{I(g)}$ is a strongly consistent estimator of $I(g)$
\end{theorem}

\begin{proof}
By (A2),
$
0\leqslant g(x_1,\ldots,x_m)
\leqslant C_g\sum_{i=1}^m x_i,
$
and hence $g(X_1,\ldots,X_m)$ is integrable under the finite Gamma mixture
model. Therefore, the strong law of large numbers for $U$-statistics \citep{Lee1990,Henze2024} gives
\begin{align*}
U_n(g)
\xrightarrow{\mathrm{a.s.}}
\mathbb{E}[g(X_1,\ldots,X_m)],
\end{align*}
where $U_n(g)$ is the  $U$-statistic defined in \eqref{eq:Us}.
Moreover, since $X_1$ has finite mean,
$
\overline X
\xrightarrow{\mathrm{a.s.}}
\mathbb{E}[X_1]=\mu>0.
$
Consequently,
\[
\widehat{I(g)}
\stackrel{\eqref{eq:Tn}}{=}
\frac{U_n(g)}{mC_g\overline X}
\xrightarrow{\mathrm{a.s.}}
\frac{\mathbb{E}[g(X_1,\ldots,X_m)]}
{mC_g\mu}
=
I(g).
\]
\end{proof}

\begin{theorem}\label{thm:asymptotic-normality}
Under the assumptions of Theorem~\ref{thm:Expectation},
\[
\sqrt{n}
\left\{
\widehat{I(g)}-I(g)
\right\}
\xrightarrow{\mathscr{D}}
N
\left(
0,\frac{\sigma_g^2}{C_g^2\mu^2}
+
\frac{\theta_g^2\sigma_X^2}
{m^2C_g^2\mu^4}
-
\frac{2\theta_g\tau_g}
{mC_g^2\mu^3}
\right),
\quad n\to\infty,
\]
where $\xrightarrow{\mathscr{D}}
$ denotes convergence in distribution,
$
\theta_g
=
\mathbb{E}[g(X_1,\ldots,X_m)],
\ 
\sigma_g^2
=
\operatorname{Var}
\left(
\mathbb{E}[g(X_1,\ldots,X_m)\mid X_1]
\right),
\ 
\tau_g
=
\operatorname{Cov}
\left(
\mathbb{E}[g(X_1,\ldots,X_m)\mid X_1],
X_1
\right),
$
and
$
\sigma_X^2=\operatorname{Var}(X_1).
$
\end{theorem}

\begin{proof}
By (A2),
$
0\leqslant g(x_1,\ldots,x_m)
\leqslant C_g\sum_{i=1}^m x_i.
$
Since the finite Gamma mixture has finite moments of all orders,
$
\mathbb{E}[g^2(X_1,\ldots,X_m)]<\infty.
$
Hence, the asymptotic normality of $U$-statistics \citep[Theorem 7.3 of][]{Hoeffding1948} yields
\[
\sqrt{n}
\left\{
U_n(g)-\theta_g
\right\}
\xrightarrow{\mathscr{D}}
N(0,m^2\sigma_g^2),
\quad n\to\infty,
\]
where $U_n(g)$ is as defined in \eqref{eq:Us}.
Moreover, the multivariate central limit theorem gives the joint
convergence
\[
\sqrt{n}
\begin{pmatrix}
U_n(g)-\theta_g
\\[0,2cm]
\overline X-\mu
\end{pmatrix}
\xrightarrow{\mathscr{D}}
N_2(\boldsymbol{0},\boldsymbol{\Sigma}),
\quad 
\boldsymbol{\Sigma}
=
\begin{pmatrix}
m^2\sigma_g^2&m\tau_g
\\[0,2cm]
m\tau_g&\sigma_X^2
\end{pmatrix}.
\]

Consider
$
h(u,v)={u}/({mC_gv}).
$
Since $\mu>0$, $h$ is continuously differentiable in a neighborhood of
$(\theta_g,\mu)^\top$, with gradient
$
\nabla h(\theta_g,\mu)
=
(
{1}/({mC_g\mu}),
-{\theta_g}/({mC_g\mu^2})
)^{\!\top}.
$
Therefore, by the multivariate delta method,
\[
\sqrt{n}
\left\{
h(U_n(g),\overline{X})-h(\theta_g,\mu)
\right\}
\stackrel{\eqref{eq:Tn}}{=}
\sqrt{n}
\left\{
\widehat{I(g)}-I(g)
\right\}
\xrightarrow{\mathscr{D}}
N(0,\sigma_I^2),
\quad 
n\to\infty,
\]
where
$
\sigma_I^2
=
[\nabla h(\theta_g,\mu)]^\top
\boldsymbol{\Sigma}
\nabla h(\theta_g,\mu).
$
\end{proof}

\subsection{Asymptotic behavior of the bias}

The exact bias expression in Corollary~\ref{cor:Bias} enables us to characterize its asymptotic behavior as $n\to\infty$. In particular, the law of large numbers, together with the dominated convergence theorem, establishes that the bias converges to zero as the sample size increases.

\begin{corollary}\label{cor:asymptotic-unbiasedness}
Under the assumptions of Corollary~\ref{cor:Bias},
$
\lim_{n\to\infty}
\operatorname{Bias}(\widehat{I(g)},I(g))=0,
$
that is, $\widehat{I(g)}$ is asymptotically unbiased as
$n\to\infty$.
\end{corollary}
\begin{proof}
Let
$
\mu_\alpha=\sum_{r=1}^{\ell}\pi_r\alpha_r
\ \text{and}\ 
\alpha_{\min}=\min_{1\leqslant r\leqslant\ell}\alpha_r>0.
$
Since $\mathbb{P}(Y_i=j)=\pi_j$, the random variables
$\alpha_{Y_1},\alpha_{Y_2},\ldots$ are iid, with
$
\mathbb{E}[\alpha_{Y_i}]=\mu_\alpha.
$
By the strong law of large numbers and the continuous mapping
theorem,
\[
\frac{n}{\sum_{i=1}^{n}\alpha_{Y_i}}
\xrightarrow{\mathrm{a.s.}}
\frac{1}{\mu_\alpha}, \quad n\to\infty.
\]

Fix any $(j_1,\ldots,j_m)^\top\in\{1,\ldots,\ell\}^m$.
Since the remaining random variables are independent of
$Y_1,\ldots,Y_m$, the strong law of large numbers yields
\[
\frac{n}{
\sum_{i=1}^{m}\alpha_{j_i}
+\sum_{i=m+1}^{n}\alpha_{Y_i}}
\xrightarrow{\mathrm{a.s.}}
\frac{1}{\mu_\alpha}.
\]
Moreover, for every $n\geqslant m$,
\[
0<
\frac{n}{
\sum_{i=1}^{m}\alpha_{j_i}
+\sum_{i=m+1}^{n}\alpha_{Y_i}}
\leqslant \frac{1}{\alpha_{\min}}.
\]
Hence, by the dominated convergence theorem,
\[
\begin{aligned}
&\mathbb{E}\left[
\frac{n}{
\sum_{i=1}^{m}\alpha_{j_i}
+\sum_{i=m+1}^{n}\alpha_{Y_i}}
-\frac{1}{\mu_\alpha}
\right]
\longrightarrow 0, \quad n\to\infty.
\end{aligned}
\]
Equivalently, for every fixed $(j_1,\ldots,j_m)^\top$,
\[
\begin{aligned}
&\sum_{j_{m+1},\ldots,j_n=1}^{\ell}
\left(\prod_{i=m+1}^{n}\pi_{j_i}\right)
\left(
\frac{n}{\sum_{i=1}^{n}\alpha_{j_i}}
-\frac{1}{\mu_\alpha}
\right)
\longrightarrow 0, \quad n\to\infty.
\end{aligned}
\]

By Corollary~\ref{cor:Bias}, the bias is a finite sum over
$j_1,\ldots,j_m$, whose remaining factors do not depend on $n$
and are finite under the stated assumptions. Since each
corresponding inner weighted sum converges to zero, we conclude
that
$
\lim_{n\to\infty}
\operatorname{Bias}(\widehat{I(g)},I(g))=0.
$
This completes the proof.
\end{proof}

\section{Exact variance and MSE of NSIIs under Gamma mixture models}
\label{sec:variance}

We now derive exact finite-sample expressions for the variance and
mean squared error (MSE) of the estimator $\widehat{I(g)}$ in \eqref{eq:Tn}. The
derivation uses the Gamma--Dirichlet representation established in
Lemma~\ref{lem:GD}. In contrast to the expectation, the second moment
involves pairs of subsets of size $m$, which are classified according
to the size of their intersection.

For $A,B\subset\{1,\ldots,n\}$ with $|A|=|B|=m$, let
$
q=|A\cap B|,
\ 
q_0=\max\{0,2m-n\},
$
so that $q_0\leqslant q\leqslant m$. Define
$
d_q=2m-q
$
and
\begin{equation}
\label{eq:Hq}
H_q\coloneqq
\mathbb E\left[
g(U_1^*,\ldots,U_m^*)
g(U_1^*,\ldots,U_q^*,U_{m+1}^*,\ldots,U_{d_q}^*)
\right],
\end{equation}
where
\[
(U_1^*,\ldots,U_n^*)^\top
=
\left(
\frac{X_1}{S},\ldots,\frac{X_n}{S}
\right)^\top,
\quad
S=\sum_{i=1}^nX_i.
\]

\begin{proposition}
\label{prop:second-moment}
Under the assumptions of Theorem~\ref{thm:Expectation},
\[
\begin{aligned}
\mathbb E[\widehat{I(g)}^2]
=
\frac{n^2}
{m^2C_g^2\binom{n}{m}}
\sum_{q=q_0}^{m}
\binom{m}{q}
\binom{n-m}{m-q}
H_q,
\end{aligned}
\]
where $H_q$ is as defined in \eqref{eq:Hq}.
\end{proposition}

\begin{proof}
By the definition of $\widehat{I(g)}$ and assumption (A1),
\[
\widehat{I(g)}
=
\frac{n}
{mC_g\binom{n}{m}}
\sum_{A\in\mathcal A_{n,m}}g(U_A),
\]
where
$
\mathcal A_{n,m}
=
\{A\subset\{1,\ldots,n\}:|A|=m\}.
$
Hence,
\[
\mathbb E[\widehat{I(g)}^2]
=
\frac{n^2}
{m^2C_g^2\binom{n}{m}^2}
\sum_{A,B\in\mathcal A_{n,m}}
\mathbb E[g(U_A)g(U_B)].
\]

For a fixed $A$, the number of subsets $B$ satisfying
$|A\cap B|=q$ is
$
\binom{m}{q}\binom{n-m}{m-q}.
$
By the exchangeability of $X_1,\ldots,X_n$ and the symmetry of $g$,
the corresponding expectation depends on $(A,B)^\top$ only through
$q=|A\cap B|$. Thus, we may take
$
A=\{1,\ldots,m\},
\
B=\{1,\ldots,q,m+1,\ldots,2m-q\},
$
which gives $H_q$. Summing over $q$ yields the result.
\end{proof}

By Lemma~\ref{lem:GD}, 
$$(U_1^*,\ldots,U_n^*)^\top\mid \{Y_1=j_1,\ldots,Y_n=j_n\}\sim \operatorname{Dir}(\alpha_{j_1},\ldots,\alpha_{j_n}).$$
Therefore, for $(U_1,\ldots,U_n)^\top\sim \operatorname{Dir}(\alpha_{j_1},\ldots,\alpha_{j_n})$, $H_q$ in \eqref{eq:Hq} can be written as
\begin{equation*}
\begin{aligned}
H_q
=
\sum_{j_1,\ldots,j_n=1}^{\ell}
\left(\prod_{i=1}^n\pi_{j_i}\right)
\mathbb E_{\operatorname{Dir}(\alpha_{j_1},\ldots,\alpha_{j_n})}
\Big[
g(U_1,\ldots,U_m)
g(U_1,\ldots,U_q,U_{m+1},\ldots,U_{d_q})
\Big].
\end{aligned}
\end{equation*}

Since
\begin{multline*}
\mathbb E_{\operatorname{Dir}(\alpha_{j_1},\ldots,\alpha_{j_n})}
\Big[
g(U_1,\ldots,U_m)
g(U_1,\ldots,U_q,U_{m+1},\ldots,U_{d_q})
\Big]
\\[0,2cm]
=
\frac{
\left(\sum_{i=1}^{d_q}\alpha_{j_i}\right)
\left(\sum_{i=1}^{d_q}\alpha_{j_i}+1\right)
}{
\left(\sum_{i=1}^{n}\alpha_{j_i}\right)
\left(\sum_{i=1}^{n}\alpha_{j_i}+1\right)
} \, 
\mathbb E_{\operatorname{Dir}(\alpha_{j_1},\ldots,\alpha_{j_{d_q}})}
\Big[
g(V_1,\ldots,V_m)
g(V_1,\ldots,V_q,V_{m+1},\ldots,V_{d_q})
\Big],
\end{multline*}
because the Dirichlet aggregation property implies that
$
\sum_{i=1}^{d_q}U_i
\sim
\operatorname{Beta}(
\sum_{i=1}^{d_q}\alpha_{j_i},
\sum_{i=d_q+1}^{n}\alpha_{j_i}
)
$
and
$$
(V_1,\ldots,V_{d_q})^\top
\stackrel{d}{=}
\left(
\frac{U_1}{\sum_{i=1}^{d_q}U_i},\ldots,
\frac{U_{d_q}}{\sum_{i=1}^{d_q}U_i}
\right)^\top
\sim
\operatorname{Dir}(\alpha_{j_1},\ldots,\alpha_{j_{d_q}})
$$
independently of $\sum_{i=1}^{d_q}U_i$, 
and, 
$
\mathbb{E}[
\{\sum_{i=1}^{d_q}U_i\}^2
]
=
{
(\sum_{i=1}^{d_q}\alpha_{j_i})
(\sum_{i=1}^{d_q}\alpha_{j_i}+1)
}/[{
\left(\sum_{i=1}^{n}\alpha_{j_i}\right)
\left(\sum_{i=1}^{n}\alpha_{j_i}+1\right)
}],
$
while $g$ is homogeneous of degree one.
Consequently, we can write $H_q$ as
\begin{multline}
\label{eq:Hq-mixture}
\!\!\!\!\!
H_q
= \!\!\!
\sum_{j_1,\ldots,j_{d_q}=1}^{\ell}
\left(\prod_{i=1}^{d_q}\pi_{j_i}\right)
\left(\sum_{i=1}^{d_q}\alpha_{j_i}\right)
\left(\sum_{i=1}^{d_q}\alpha_{j_i}+1\right)
\mathbb E_{\operatorname{Dir}(\alpha_{j_1},\ldots,\alpha_{j_{d_q}})}
\Big[
g(V_1,\ldots,V_m)
g(V_1,\ldots,V_q,V_{m+1},\ldots,V_{d_q})
\Big]
\\[0,2cm]
\times
\sum_{j_{d_q+1},\ldots,j_n=1}^{\ell}
\left(\prod_{i=d_q+1}^{n}\pi_{j_i}\right)
\frac{
1
}{
\left(\sum_{i=1}^{n}\alpha_{j_i}\right)
\left(\sum_{i=1}^{n}\alpha_{j_i}+1\right)
}.
\end{multline}

\begin{remark}
For any generator $g$, the expectation in \eqref{eq:Hq-mixture}
admits the general representation
\begin{multline*}
\mathbb{E}_{\operatorname{Dir}(\alpha_{j_1},\ldots,\alpha_{j_{d_q}})}
\Big[
g(V_1,\ldots,V_m)
g(V_1,\ldots,V_q,V_{m+1},\ldots,V_{d_q})
\Big]
\\[0,2cm]
=
\frac{
\Gamma\!\left(\sum_{i=1}^{d_q}\alpha_{j_i}\right)
}{
\prod_{i=1}^{d_q}\Gamma(\alpha_{j_i})
}
\int_{\Delta_{d_q-1}}
g(v_1,\ldots,v_m)
g(v_1,\ldots,v_q,v_{m+1},\ldots,v_{d_q})
\prod_{i=1}^{d_q}v_i^{\alpha_{j_i}-1}
\,{\rm d}\boldsymbol v,
\end{multline*}
where
\[
\Delta_{d_q-1}
=
\left\{
(v_1,\ldots,v_{d_q})\in(0,1)^{d_q}:
\sum_{i=1}^{d_q}v_i=1
\right\},
\quad d_q=2m-q.
\]
For particular choices of $g$, the above integral may simplify further.
In general, however, the overlap between the two sets of arguments
prevents a simple closed-form expression.
The above integral therefore provides a
general expression that can be evaluated analytically or numerically for the
generators considered.
\end{remark}

Hence, we obtain the following exact representation.
\begin{theorem}
\label{thm:variance}
Under the assumptions of Theorem~\ref{thm:Expectation},
{
\begin{align*}
\operatorname{Var}[\widehat{I(g)}]
=
\frac{n^2}{m^2C_g^2}
\left\{
\frac{1}{\binom{n}{m}}
\sum_{q=q_0}^{m}
\binom{m}{q}
\binom{n-m}{m-q}
H_q
-
\left\{
\sum_{j_1,\ldots,j_n=1}^{\ell}
\left(
\prod_{i=1}^{n}\pi_{j_i}
\right) 
\left(
{\sum_{i=1}^m \alpha_{j_i}} 
\right)
{
\mathbb{E}_{\operatorname{Dir}
(\alpha_{j_1},\ldots,\alpha_{j_m})}
\!\left[
g(V_1,\ldots,V_m)
\right]
\over 
\sum_{i=1}^n \alpha_{j_i}
}
\right\}^2
\right\},
\end{align*}
}
where
$
(V_1,\ldots,V_m)^\top 
\sim
\operatorname{Dir}
(
\alpha_{j_1},
\ldots,
\alpha_{j_m}
)
$,
$
q_0=\max\{0,2m-n\},
\
d_q=2m-q,
$
and $H_q$ is as given in \eqref{eq:Hq-mixture}.
\end{theorem}
\begin{proof}
The result follows from
$
\operatorname{Var}[\widehat{I(g)}]
=
\mathbb E[\widehat{I(g)}^2]
-
\{\mathbb E[\widehat{I(g)}]\}^2,
$
together with Proposition~\ref{prop:second-moment} and
Theorem \ref{thm:Expectation}.
\end{proof}

As in the preceding results, the common rate parameter $\lambda$
does not appear in $\operatorname{Var}[\widehat{I(g)}]$. Thus, the variance of the
normalized estimator depends on the Gamma mixture model through the
mixture weights and shape parameters.

The MSE follows from the variance given in Theorem \ref{thm:variance} and the exact bias obtained in Corollary \ref{cor:Bias}.

\begin{corollary}
\label{cor:MSE}
Under the assumptions of Theorem~\ref{thm:variance},
{
\begin{multline*}
\!\!\!\!\!\!
\operatorname{MSE}[\widehat{I(g)}]
=
\frac{1}{m^2C_g^2}
\left\{
\frac{n^2}{\binom{n}{m}} \!
\sum_{q=q_0}^{m}
\binom{m}{q}
\binom{n-m}{m-q}
H_q
-
\left\{
n \!
\sum_{j_1,\ldots,j_n=1}^{\ell} \!
\left(
\prod_{i=1}^{n}\pi_{j_i}
\right) 
\left(
{\sum_{i=1}^m \alpha_{j_i}} 
\right)
{
\mathbb{E}_{\operatorname{Dir}
(\alpha_{j_1},\ldots,\alpha_{j_m})}
\!\left[
g(V_1,\ldots,V_m)
\right]
\over 
\sum_{i=1}^n \alpha_{j_i}
}
\right\}^2
\right.
\\[0,2cm]
+
\left.
\left\{
\sum_{j_1,\ldots,j_m=1}^{\ell} \!
\left(
\prod_{i=1}^{m}\pi_{j_i}
\right)  \!
\left(
\sum_{i=1}^{m}\alpha_{j_i}
\right)
\mathbb{E}_{\operatorname{Dir}
(\alpha_{j_1},\ldots,\alpha_{j_m})}
\!\left[
g(V_1,\ldots,V_m)
\right] \!\!
\sum_{j_{m+1},\ldots,j_n=1}^{\ell} \!
\left(
\prod_{i=m+1}^{n}\pi_{j_i}
\right)  \!
\Bigg(
{n\over \sum_{i=1}^n \alpha_{j_i}} 
-
\frac1{\sum_{r=1}^{\ell}\pi_r\alpha_r}
\Bigg)
\right\}^2
\right\}, 
\end{multline*}
} \noindent
where
$
(V_1,\ldots,V_m)^\top 
\sim
\operatorname{Dir}
(
\alpha_{j_1},
\ldots,
\alpha_{j_m}
)
$,
$
q_0=\max\{0,2m-n\},
\
d_q=2m-q,
$
and $H_q$ is as given in \eqref{eq:Hq-mixture}.
\end{corollary}

\begin{corollary}
\label{cor:variance-gamma}
Consider the special case in which the finite Gamma mixture consists of
a single Gamma distribution, namely,
$
\ell=1,
\ 
X\sim\operatorname{Gamma}(\alpha,\lambda).
$
Then,
{\small
\begin{align*}
\!\!\!\!\!
&\operatorname{Var}[\widehat{I(g)}]
=
\frac{n}
{m^2(n\alpha+1) C_g^2 \binom{n}{m}}
\sum_{q=q_0}^{m}
\binom{m}{q}
\binom{n-m}{m-q}
\\[0,2cm]
&\times 
\Bigg\{
d_q(d_q\alpha+1)
\mathbb E_{\operatorname{Dir}_{d_q}(\alpha,\ldots,\alpha)}
\Big[
g(V_1,\ldots,V_m)
g(V_1,\ldots,V_q,V_{m+1},\ldots,V_{d_q})
\Big]
-
\frac{m^2(n\alpha+1)}{n}
\left\{
\mathbb E_{\operatorname{Dir}_m(\alpha,\ldots,\alpha)}
\left[
g(V_1,\ldots,V_m)
\right]
\right\}^{2}
\Bigg\},
\end{align*}
}\noindent
where
$
(V_1,\ldots,V_m)^\top 
\sim
\operatorname{Dir}_m
(
\alpha,
\ldots,
\alpha
)
$,
$
(V_1,\ldots,V_{d_q})^\top
\sim
\operatorname{Dir}_{d_q}(\alpha,\ldots,\alpha)
$
$q_0=\max\{0,2m-n\}$ and $d_q=2m-q$.

Consequently,
$
\operatorname{MSE}[\widehat{I(g)}]
=
\operatorname{Var}[\widehat{I(g)}].
$
\end{corollary}
\begin{proof}
Since $\ell=1$, we have $j_1=\cdots=j_n=1$, $\pi_1=1$, and
$\alpha_1=\alpha$. The result then follows directly from
Theorem~\ref{thm:variance} and the identity
$
\sum_{q=q_0}^{m}
\binom{m}{q}
\binom{n-m}{m-q}
=
\binom{n}{m}.
$
\end{proof}

\begin{remark}
When $m=n$, the only possible value of $q$ is
$q=q_0=n=d_q$. Therefore, by Corollary \ref{cor:variance-gamma},
\begin{align*}
\operatorname{Var}[\widehat{I(g)}]
&=
\frac{1}{C_g^2}
\left\{
\mathbb E_{\operatorname{Dir}_{n}(\alpha,\ldots,\alpha)}
\left[g(V_1,\ldots,V_n)^2\right]
-
\left[
\mathbb E_{\operatorname{Dir}_{n}(\alpha,\ldots,\alpha)}
\left[g(V_1,\ldots,V_n)\right]
\right]^2
\right\}
\\[0,2cm]
&=
\operatorname{Var}_{\operatorname{Dir}_{n}(\alpha,\ldots,\alpha)}\left[{g(V_1,\ldots,V_n)\over C_g}\right].
\end{align*}
Thus, when $m=n$, the variance reduces to the variance of
$g(V_1,\ldots,V_n)/C_g$ under the
$\operatorname{Dir}_{n}(\alpha,\ldots,\alpha)$ distribution.
\end{remark}

\section{Simulation study}
\label{Simulation study}


\subsection{Evaluating the Bias and RMSE of estimates}

In this subsection, we present a Monte Carlo simulation study to evaluate the estimator $\widehat{I(g)}$ \eqref{eq:Tn} for the NSIIs \eqref{eq:Ig-Dir} with $m=2$. 
Set a Gamma mixture \eqref{eq:gamma-mixture} with 
$\ell=2$ and
$\boldsymbol{\theta}=(\pi_1,\pi_2,\alpha_1,\alpha_2,\lambda)^\top$.

We conducted $M=10,000$ Monte Carlo replications, generating samples $\mathbf{X}^{(m)} = \{X_1^{(m)}, X_2^{(m)},\cdots, X_n^{(m)}\}$ drawn from the $\operatorname{GM}(\boldsymbol{\theta})$ considering for the simulations the true population's parameter $\boldsymbol{\theta}=(0.6,0.4,0.5,5,1)^\top$, and the sample sizes $n\in\{10, 20, 50, 100, 200\}$.
Moreover, we initially chose the normalized scale-invariant inequality indices Gini, SCV, Power, Shannon entropy, Product, Symmetric, Maximum, and Minimum. 

Aiming to study the Monte Carlo estimates, we compute the Bias and root mean squared error (RMSE) of each estimated value with respect to its true population values. 
Figure \ref{fig:sim_m2} presents the Bias and RMSE for each estimated parameter. 
Observe that when the sample size $n$ increases, both the Bias and the RMSE of the Monte Carlo simulations generally decrease. Table \ref{tab:mix_m2} complements this analysis by showing that the Monte Carlo biases align with the exact biases (Corollary~\ref{cor:Bias}).

\begin{figure}[H]
    \centering
\includegraphics[width=1.0\linewidth]{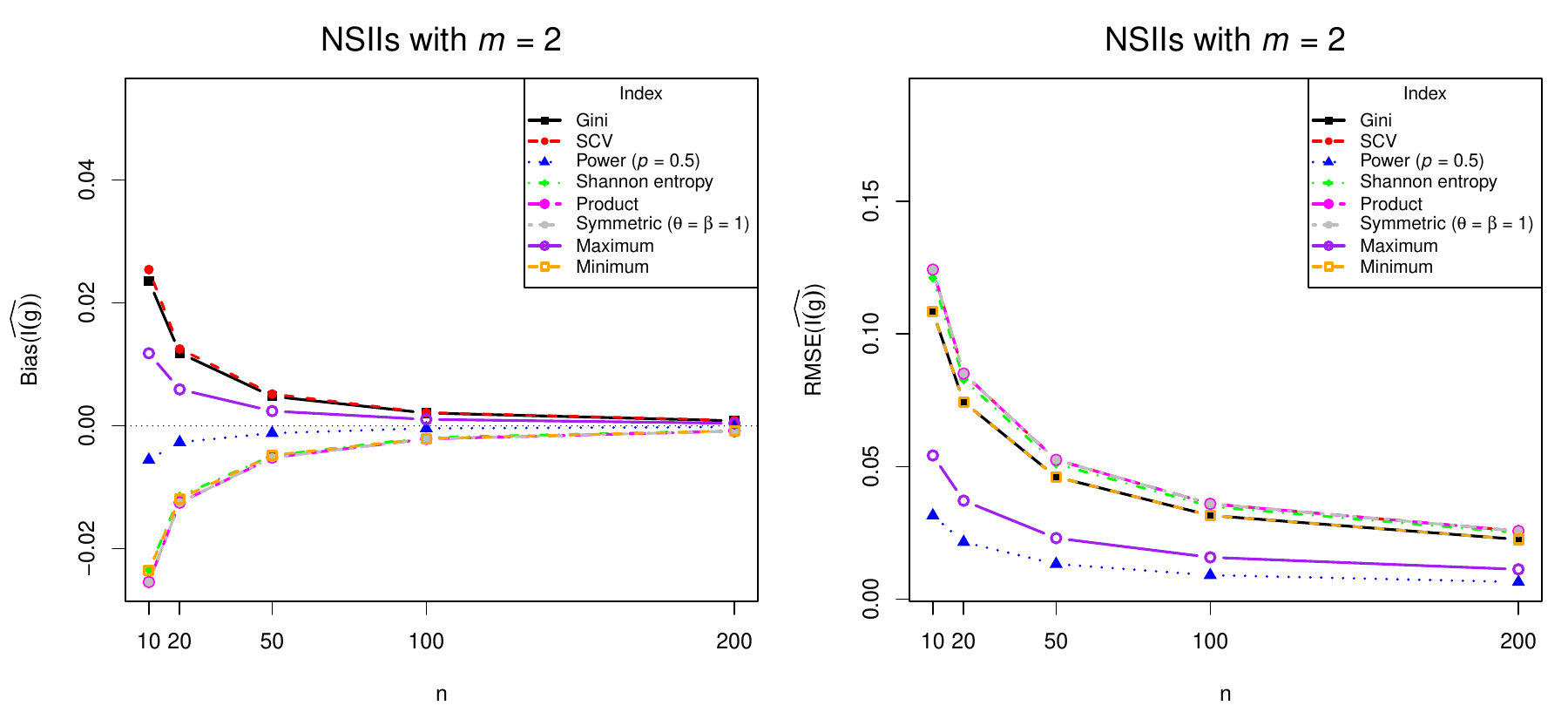}
    \caption{Bias (left) and RMSE (right) for the estimator $\widehat{I(g)}$ \eqref{eq:Tn} for the NSIIs \eqref{eq:Ig-Dir} based on Monte Carlo simulations of $\operatorname{GM}(\boldsymbol{\theta})$. }
    \label{fig:sim_m2}
\end{figure}

\begin{table}[H]
\centering
\caption{{Exact (Corollary~\ref{cor:Bias}) and Monte Carlo (MC) bias and RMSE of $\widehat{I(g)}$ for the NSIIs with $m=2$, under $X\sim\mathrm{GM}(\boldsymbol{\theta})$ with $\boldsymbol{\theta}=(0.6,0.4,0.5,5,1)^\top$ and $M=10{,}000$ replications.}}
\label{tab:mix_m2}
\footnotesize
\setlength{\tabcolsep}{2.2pt}
{
\begin{tabular}{llcrrrrrrrrrr}
\hline
& & & \multicolumn{2}{c}{$n=10$} & \multicolumn{2}{c}{$n=20$} & \multicolumn{2}{c}{$n=50$} & \multicolumn{2}{c}{$n=100$} & \multicolumn{2}{c}{$n=200$} \\
\cline{4-5}\cline{6-7}\cline{8-9}\cline{10-11}\cline{12-13}
Index & $I(g)$ & & Bias & RMSE & Bias & RMSE & Bias & RMSE & Bias & RMSE & Bias & RMSE \\
\hline
Gini & $0.6068$ & MC & $0.0236$ & $0.1084$ & $0.0119$ & $0.0744$ & $0.0048$ & $0.0461$ & $0.0021$ & $0.0316$ & $0.0008$ & $0.0226$ \\
 & & Exact & $0.0242$ &  & $0.0125$ &  & $0.0050$ &  & $0.0025$ &  & $0.0013$ &  \\
SCV & $0.4841$ & MC & $0.0254$ & $0.1243$ & $0.0125$ & $0.0851$ & $0.0052$ & $0.0526$ & $0.0021$ & $0.0360$ & $0.0009$ & $0.0257$ \\
 & & Exact & $0.0260$ &  & $0.0134$ &  & $0.0054$ &  & $0.0027$ &  & $0.0014$ &  \\
Power ($p=0.5$) & $0.8999$ & MC & $-0.0055$ & $0.0316$ & $-0.0027$ & $0.0215$ & $-0.0012$ & $0.0133$ & $-0.0004$ & $0.0091$ & $-0.0002$ & $0.0065$ \\
 & & Exact & $-0.0057$ &  & $-0.0030$ &  & $-0.0012$ &  & $-0.0006$ &  & $-0.0003$ &  \\
Shannon entropy & $0.5651$ & MC & $-0.0234$ & $0.1211$ & $-0.0115$ & $0.0828$ & $-0.0049$ & $0.0511$ & $-0.0019$ & $0.0350$ & $-0.0008$ & $0.0250$ \\
 & & Exact & $-0.0240$ &  & $-0.0124$ &  & $-0.0050$ &  & $-0.0025$ &  & $-0.0013$ &  \\
Product & $0.5159$ & MC & $-0.0254$ & $0.1243$ & $-0.0125$ & $0.0851$ & $-0.0052$ & $0.0526$ & $-0.0021$ & $0.0360$ & $-0.0009$ & $0.0257$ \\
 & & Exact & $-0.0260$ &  & $-0.0134$ &  & $-0.0054$ &  & $-0.0027$ &  & $-0.0014$ &  \\
Symmetric subcase ($\theta=\beta=1$) & $0.5159$ & MC & $-0.0254$ & $0.1243$ & $-0.0125$ & $0.0851$ & $-0.0052$ & $0.0526$ & $-0.0021$ & $0.0360$ & $-0.0009$ & $0.0257$ \\
 & & Exact & $-0.0260$ &  & $-0.0134$ &  & $-0.0054$ &  & $-0.0027$ &  & $-0.0014$ &  \\
Maximum & $0.8034$ & MC & $0.0118$ & $0.0542$ & $0.0059$ & $0.0372$ & $0.0024$ & $0.0231$ & $0.0010$ & $0.0158$ & $0.0004$ & $0.0113$ \\
 & & Exact & $0.0121$ &  & $0.0062$ &  & $0.0025$ &  & $0.0013$ &  & $0.0006$ &  \\
Minimum & $0.3932$ & MC & $-0.0236$ & $0.1084$ & $-0.0119$ & $0.0744$ & $-0.0048$ & $0.0461$ & $-0.0021$ & $0.0316$ & $-0.0008$ & $0.0226$ \\
 & & Exact & $-0.0242$ &  & $-0.0125$ &  & $-0.0050$ &  & $-0.0025$ &  & $-0.0013$ &  \\
\hline
\end{tabular}
}
\end{table}

\subsection{Higher-order $m$, weight configurations and shape influences}

The estimator $\widehat{I(g)}$ in \eqref{eq:Tn} depends on the value of $m$ and on a configuration, either $(j,k)^\top$ or $(a_1,\ldots,a_m)^\top$, which specifies different generating functions $g$, and thus possible different $I(g)$. 
As in the previous subsection, the biases are of order $1/n$, and the RMSE decreases as $1/n^{1/2}$.
Figure \ref{fig:sim_m} presents the Monte Carlo bias and RMSE for the estimator $\widehat{I(g)}$ \eqref{eq:Tn} for the $m$th Gini, Extended $m$th Gini, and Linear OS indices based on Monte Carlo simulations of $\operatorname{GM}(\boldsymbol{\theta})$.
Figure \ref{fig:sim_m} shows that lower biases depend not only on the $m$ values but also on the configurations chosen for each index.

\begin{figure}[H]
    \centering
\includegraphics[width=1.0\linewidth]{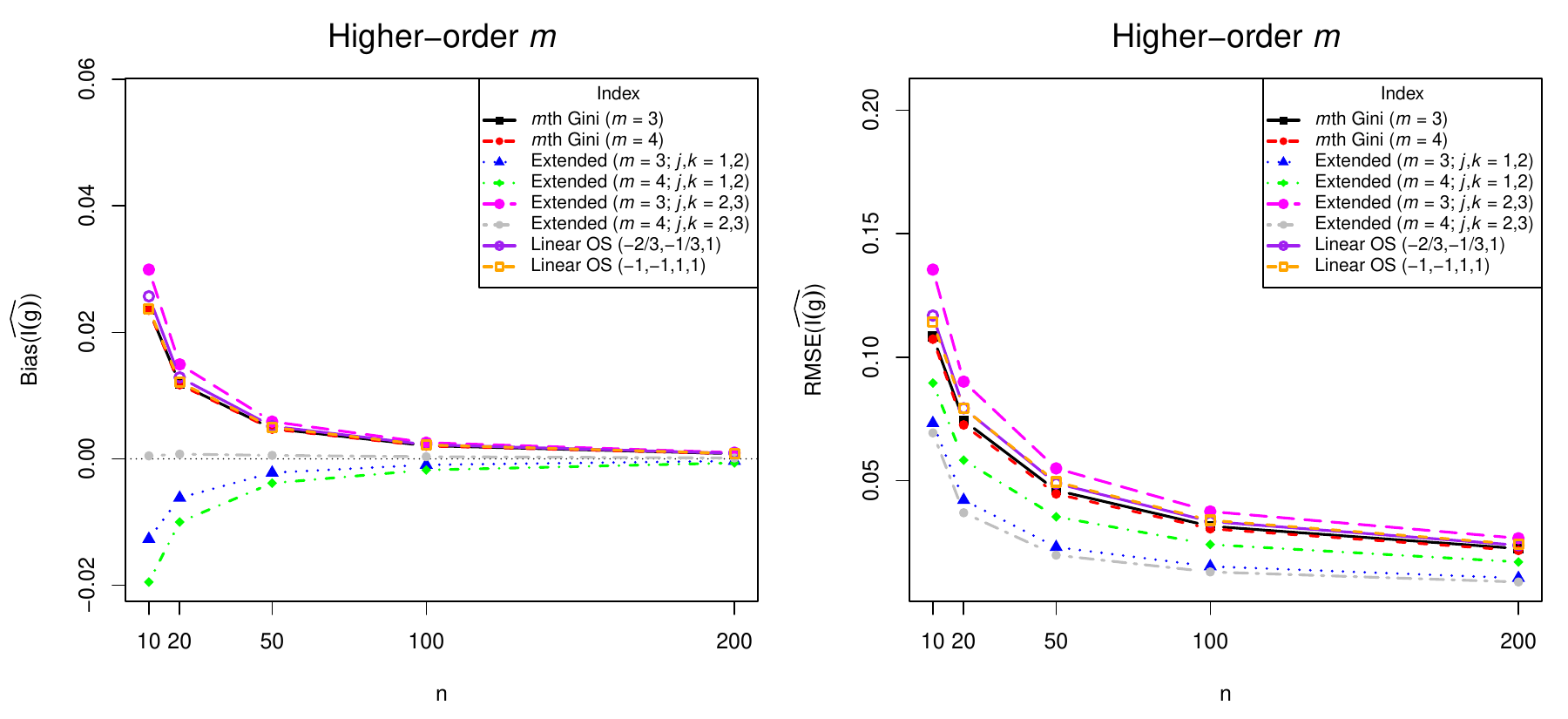}
    \caption{Bias (left) and RMSE (right) for the estimator $\widehat{I(g)}$ \eqref{eq:Tn} for $m$th Gini, Extended $m$th Gini, and Linear OS indices based on Monte Carlo simulations of $\operatorname{GM}(\boldsymbol{\theta})$.}
    \label{fig:sim_m}
\end{figure}

Under the same configurations of the $m$th Gini, Extended $m$th Gini, and Linear OS indices, we also analyze the impact of increasing the shape parameter of one of the components. 
The previous results showed that a sample size of $n=100$ yields good outcomes. We now keep this sample size to study the impacts of changing $\alpha_2$.
Consider a $\operatorname{GM}(\boldsymbol{\theta}_1)$, where $\boldsymbol{\theta}=(0.6,0.4,0.5,\alpha_2,1)^\top$, $\alpha_2\in\{5, 10, 20, 50, 100\}$.
Figure \ref{fig:sim_shape} shows that the Monte Carlo bias tends to increase for large values of $n$, while in general it stabilizes for small values. On the other hand, the RMSE slightly increased with the shape parameter.

\begin{figure}[H]
    \centering
\includegraphics[width=1.0\linewidth]{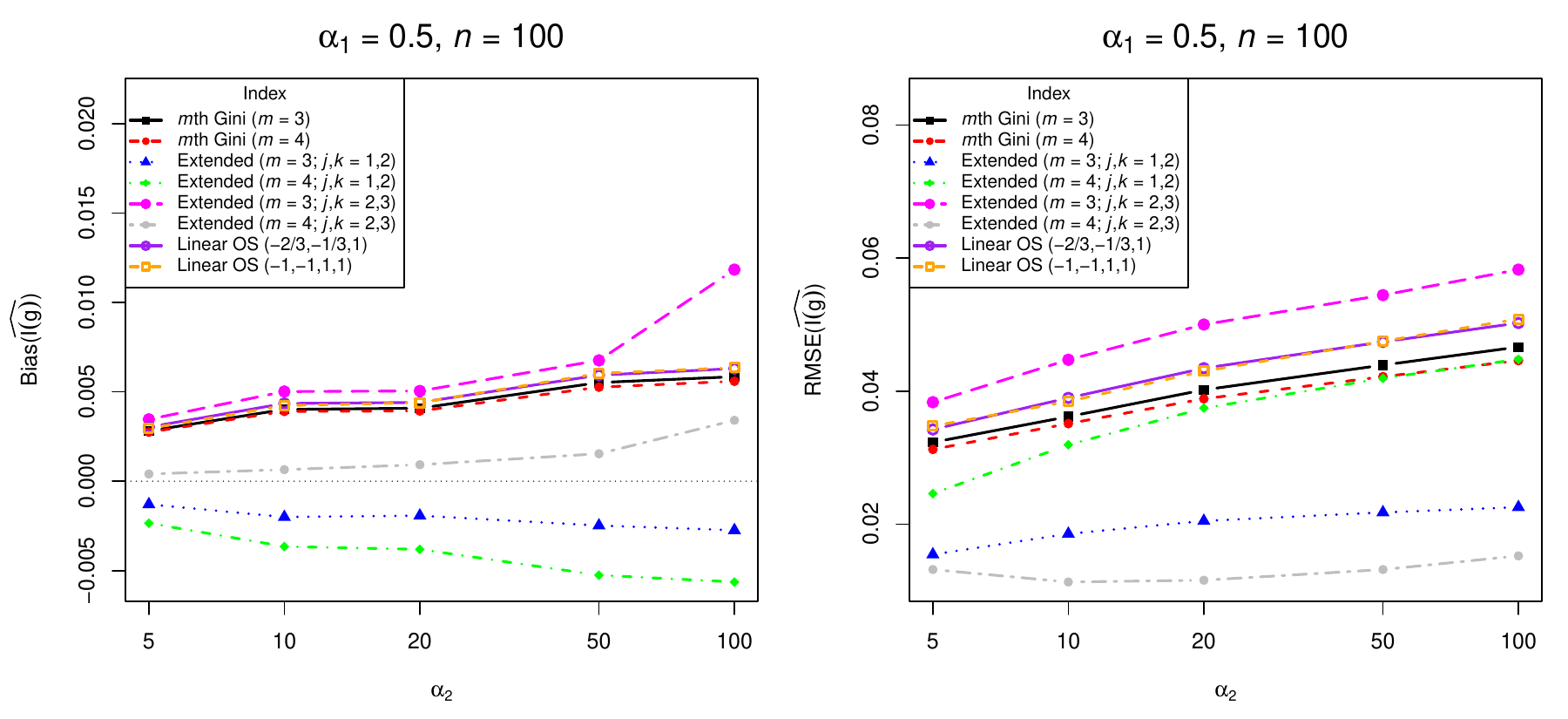}
    \caption{Bias (left) and RMSE (right) for the estimator $\widehat{I(g)}$ \eqref{eq:Tn} for $m$th Gini, Extended $m$th Gini and Linear OS indices based on Monte Carlo simulations of $\operatorname{GM}((0.6,0.4,0.5,\alpha_2,1)^\top)$, varying $\alpha_2\in\{5, 10, 20, 50, 100\}$ on the x-axis. }
    \label{fig:sim_shape}
\end{figure}

\subsection{Assessing asymptotic normality}

In this subsection, we assess the asymptotic normality of the estimator $\widehat{I(g)}$ established in Theorem~\ref{thm:asymptotic-normality}. 
Consider 
$$Z_n=\sqrt{n}\{\widehat{I(g)}-I(g)\}/\sigma_I,$$
where $\sigma_I^2$ is the asymptotic variance in
Theorem~\ref{thm:asymptotic-normality}, computed by numerical integration.
Figure~\ref{fig:normality} presents boxplots, histograms, and qq-plots of the statistics $Z_n^{(m)}$ $(m=1,\cdots, 10000)$ based on the Monte Carlo simulations with sample size $n=200$ reported in Figures \ref{fig:sim_m2} and \ref{fig:sim_m}. 
The plots show a symmetric empirical distribution of the estimates, a good fit to the normal PDF, and well-behaved normal residuals.
The study was reproduced for all indices reported in Figures \ref{fig:sim_m2} and \ref{fig:sim_m}, but, for simplicity, Figure \ref{fig:normality} shows the results for the Gini, Product, and $m$th indices.

\begin{figure}[H]
    \centering
\includegraphics[width=1.0\linewidth]{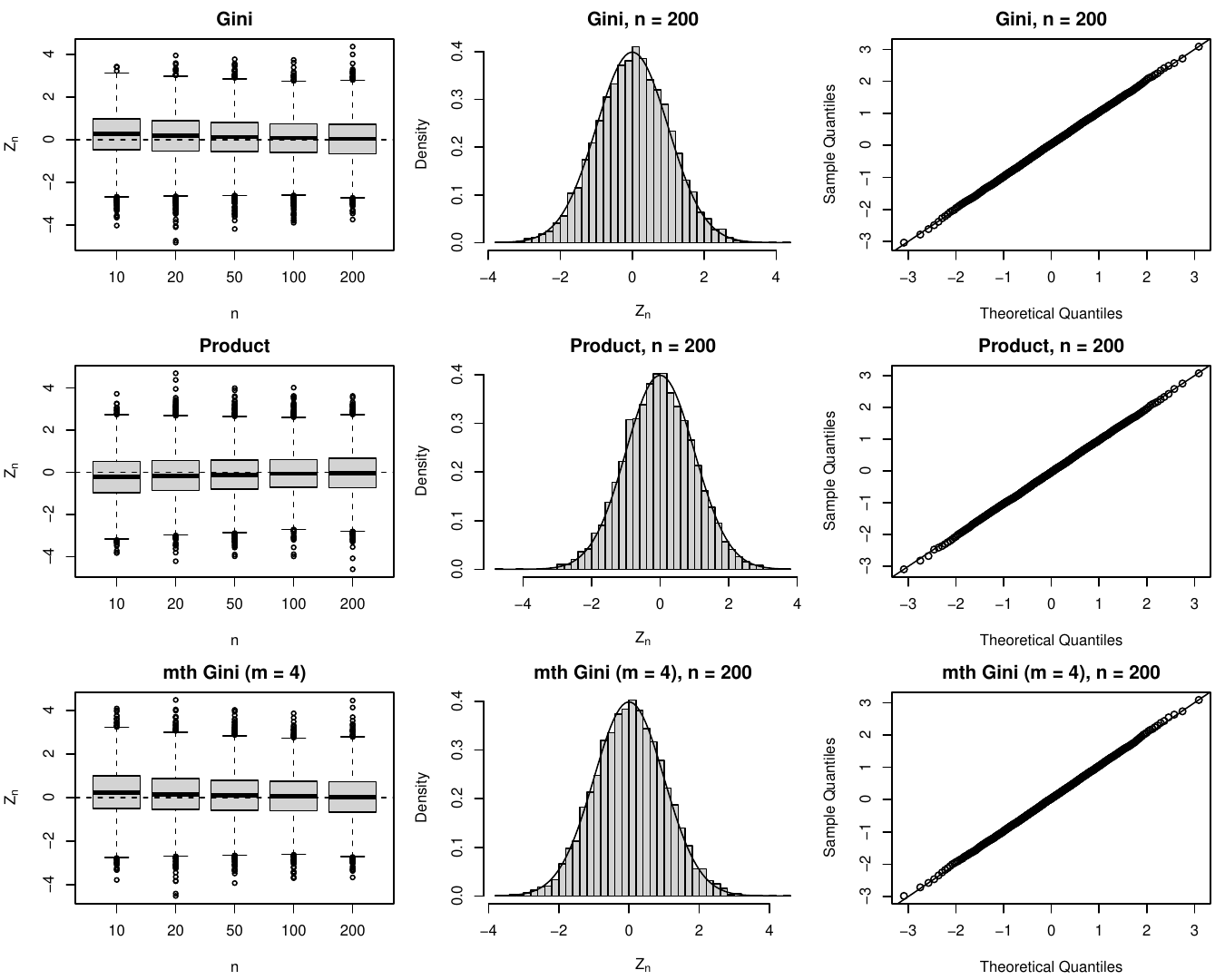}
    \caption{Boxplot (left), histograms (middle), and normal QQ-plots for the Monte Carlo estimates $\widehat{I(g)}^{(m)}$.  Results for Gini are shown on top, Product on middle, and $m$th Gini on the bottom.}
    \label{fig:normality}
\end{figure}

 \subsection{Analysis of the estimator in GM populations with non-identical rates}

Previous sections describe the behavior of the estimator $\widehat{I(g)}$ in \eqref{eq:Tn} only under the GM with identical-rate parameters. In this subsection, we discuss its natural extensions to non-identical rates, that is, 
\begin{equation}
\label{eq:gamma-mixture}
f_X(x;\boldsymbol{\theta})=\sum_{r=1}^{2}
\pi_r \, 
\frac{\lambda^{\alpha_r}}{\Gamma(\alpha_r)} \, 
x^{\alpha_r-1}\text{e}^{-\lambda_r x},
\quad x>0,
\end{equation}
where $(\pi_1,\pi_2,\alpha_1,\alpha_2)^\top=(0.6,0.4,0.5,5)^\top$ and
$(\lambda_1,\lambda_2)^\top=(1,0.2)^\top$.
In this numerical study, the $m$th Gini, Extended $m$th Gini, and Linear OS indices are computed using numerical integration.

{The proofs of
Theorems~\ref{thm:strong-consistency} and~\ref{thm:asymptotic-normality} rely on the identical rate assumption. Therefore, we can expect that $\widehat{I(g)}$ remains strongly consistent and asymptotically normal.} The results shown in Figure \ref{fig:shapes} support this extension, indicating that the estimator converges even in the case of more general populations.

\begin{figure}[H]
    \centering
\includegraphics[width=1.0\linewidth]{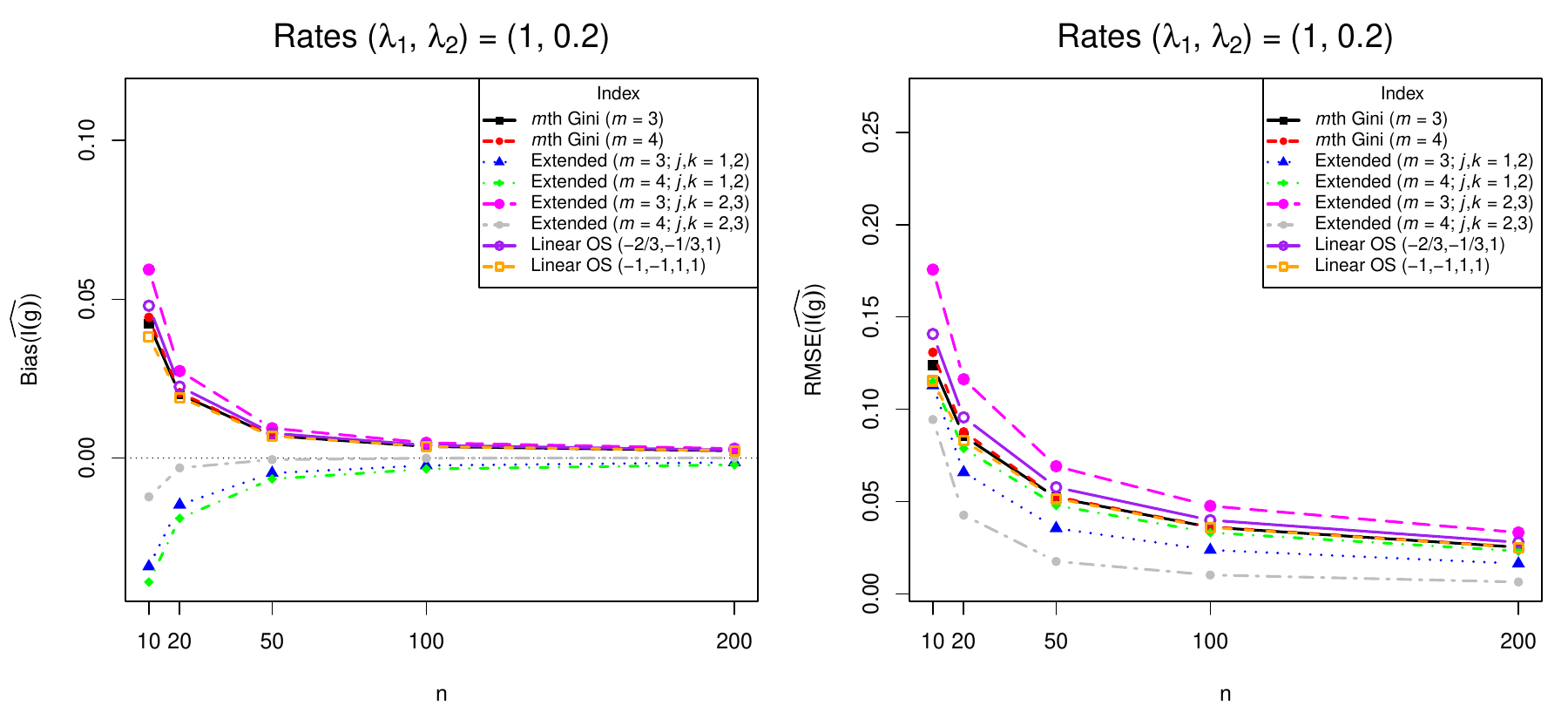}
    \caption{Bias (left) and RMSE (right) for the estimator $\widehat{I(g)}$ \eqref{eq:Tn} for $m$th Gini, Extended $m$th Gini, and Linear OS indices based on Monte Carlo simulations of $\operatorname{GM}$ populations with non-identical rates.}
    \label{fig:shapes}
\end{figure}

\section{Concluding remarks}
\label{sect-5}

In this paper, we have developed a general finite-sample framework for estimating normalized scale-invariant inequality indices (NSIIs) under finite Gamma mixture models with a common rate parameter. The proposed framework applies to a broad class of indices generated by nonnegative measurable symmetric functions that are homogeneous of degree one and satisfy the required normalization condition. This class includes, among others, the Gini coefficient, the $m$th Gini index, the extended $m$th Gini index, linear order-statistic inequality indices, and several entropy- and dispersion-based measures.

The Gamma--Dirichlet mixture representation provides the main probabilistic tool underlying our analysis. It allows the population index and the expectation of its natural $U$-statistic estimator to be expressed in terms of Dirichlet expectations, mixture weights, and component shape parameters. Based on these representations, we derived an exact finite-sample expression for the bias of the estimator. We also established that the bias converges to zero as the sample size increases, thereby proving the asymptotic unbiasedness of the estimator. Furthermore, when the mixture degenerates to a single Gamma distribution, the finite-sample bias vanishes identically, recovering the exact unbiasedness property previously established for NSIIs under the Gamma model. Although the bias admits an exact representation, its sign cannot, in general, be determined solely from the nonnegativity of the function generating the index.

We also derived exact finite-sample expressions for the variance and mean squared error of the estimator. The variance calculation accounts for the dependence among the terms of the $U$-statistic by classifying pairs of subsets according to the size of their intersection. These results provide a unified characterization of the finite-sample accuracy of NSII estimators and show that, after normalization, the common rate parameter does not affect their bias or variance. Instead, their finite-sample behavior is determined by the mixture weights, component shape parameters, and the structure of the function generating the index.

In addition to these exact finite-sample results, we established strong consistency and asymptotic normality of the estimator. Thus, although finite Gamma mixtures may introduce finite-sample bias, this bias vanishes asymptotically, and the estimator retains standard large-sample properties.

The numerical illustrations and simulation study complement the theoretical findings by examining the finite-sample behavior of the estimator under different Gamma mixture configurations. In particular, the numerical results illustrate how the mixture weights, component shape parameters, and sample size affect the bias, variance, and mean squared error. The simulations also provide empirical evidence of the asymptotic decay of the bias and the consistency of the estimator as the sample size increases.

Several directions merit further investigation. An important extension would be to study finite Gamma mixtures with component-specific rate parameters, for which the Gamma--Dirichlet representation used here no longer applies in the same form. It would also be of interest to develop analogous finite-sample results for more general mixture models and to investigate the behavior of the proposed estimators under model misspecification. Further research could address alternative estimation procedures when the mixture parameters are unknown and must be estimated from the data, as well as the construction of confidence intervals and finite-sample inference procedures for NSIIs under mixture models.

Overall, our results extend the finite-sample estimation theory for NSIIs from the single Gamma model to a broader class of heterogeneous Gamma populations. The proposed framework provides a unified basis for quantifying finite-sample bias, variance, and mean squared error for a wide range of scale-invariant inequality measures, while clarifying the role of mixture heterogeneity in their estimation.



\section*{Declarations}

\paragraph*{Ethics Approval}
Not applicable.


\paragraph*{Disclosure statement}
There are no conflicts of interest to disclose.


\paragraph*{Author Contributions Statement}
All authors contributed equally to the conception and design of the study, data analysis, interpretation of the results, manuscript preparation, and revision. All authors read and approved the final version of the manuscript.


\bibliographystyle{apalike}


\end{document}